\documentclass{amsart}
\newtheorem{theorem}{Theorem}[section]

\theoremstyle{definition}

\newtheorem{thm}[theorem]{Theorem}
\newtheorem{prop}[theorem]{Proposition}
\newtheorem{lem}[theorem]{Lemma}
\newtheorem{defn}[theorem]{Definition}

\theoremstyle{remark}

\newtheorem{rmk}[theorem]{Remark}
\numberwithin{equation}{section}
\newcommand{\bbP}{{\bold \mathcal P}}

\begin{document}

\title[Best Approximation Estimates for the Allen-Cahn Equations]{Stability Analysis and Best Approximation Error Estimates of Discontinuous Time-Stepping Schemes for the Allen-Cahn Equation}

\author{Konstantinos Chrysafinos}
\address{Department of Mathematics, School of Applied
Mathematics and Physical Sciences, National Technical University
of Athens, Zografou Campus, Athens 15780, Greece.}
\email{chrysafinos@math.ntua.gr}


\subjclass[2000]{Primary 65M12, 65M60;}

\date{}


\keywords{Allen-Cahn Equations, Best Approximation Error Estimates, Discontinuous Time-Stepping Schemes}

\maketitle

\begin{abstract}
Fully-discrete approximations of the Allen-Cahn equation are considered. In particular, we consider schemes of arbitrary order based on a discontinuous Galerkin (in time) approach combined with standard conforming finite elements (in space). We prove best approximation a-priori error estimates, with constants depending polynomially upon $(1/\epsilon)$ by circumventing Gr\"onwall Lemma arguments. We also prove that these schemes are unconditionally stable under minimal regularity assumptions on the given data.
The key feature of our approach is an appropriate duality argument, combined with a boot-strap technique. 
\end{abstract}

\newcounter{fgh}[section]
\renewcommand{\theequation}{\thesection.\arabic{equation}}
\newcommand{\bu}{{\bf u}}
\newcommand{\bv}{{\bf v}}
\newcommand{\bw}{{\bf w}}
\newcommand{\bg}{{\bf g}}
\newcommand{\bz}{{\bf z}}
\newcommand{\by}{{\bf y}}
\newcommand{\bs}{{\bf s}}
\newcommand{\bbf}{{\bf f}}
\newcommand{\bx}{{\bf x}}
\newcommand{\bphi}{{\mbox{\boldmath $\phi$}}}
\newcommand{\blambda}{{\mbox{\boldmath $\lambda$}}}

\section{Introduction}
The Allen-Cahn equation is a parameter dependent parabolic semi-linear PDE of the form
\begin{equation} \label{eqn:ac}
\left\{\begin{array}{rll}
\displaystyle u_t - \Delta u + \frac{1}{\epsilon^2} (u^3-u) &= f &\quad\mbox{ in $(0,T) \times \Omega$} \\
\displaystyle u &= 0
& \quad\mbox{ on $(0,T) \times \Gamma$} \\
\displaystyle u(0,x) &=u_0 & \quad\mbox{ in $\Omega$.}
\end{array} \right.
\end{equation}

Here, $\Omega$ denotes a bounded domain in ${\mathbb R}^d$,  $d=2,3$ with
Lipschitz boundary $\Gamma$, $u_0$ and $f$ denote
the initial data and the forcing term respectively. The principal difficulty involved, concerns the parameter
$0< \epsilon <<1$ appearing in the model problem, which is typically very small and comparable to the size of the
time and space discretization parameters, $\tau$, $h$ respectively. The Allen-Cahn equation can be viewed as the simplest
phase field model, and was introduced in \cite{AlCa79}. 

The Allen-Cahn equation represents a model problem which possess structural difficulties which significantly complicate the numerical analysis of any potential scheme.
In particular, the physical phenomena modeled by
the above system, posses complex dynamics for realistic values of
the parameter $\epsilon <<1$. For instance, we note that the natural
energy norms $\|.\|_{L^{\infty}[0,T;L^2(\Omega)]}$, $\|.\|_{L^2[0,T;H^1_0(\Omega)]}$ imposed by the structure of our problem scale
differently in terms of the parameter $\epsilon$, compared to the norm of  $\|.\|_{L^4[0,T;L^4(\Omega)]}$ that naturally arises from the nonlinear term. In addition, the presence of $L^2[0,T;L^2(\Omega)]$ norm with the ``wrong sign'' poses a substantial
difficulty in the analysis as well as in the numerical analysis of fully-discrete schemes for such problem.
Classical techniques based on Gronwall's type inequalities typically fail, since they introduce constants depending on quantities of ${\rm exp}(1/\epsilon)$. This problem was first circumvented in the works of
\cite{AlFu93,Che94,MoSc95} by developing uniform bounds of the principle eigenvalue of the linearized Allen-Cahn operator, i.e., bounds for the quantity $ \inf_{0 \neq v \in H^1(\Omega)} \frac{ \|\nabla v\|^2_{L^2(\Omega)} + \left ( (3u^2(t)-1)v,v \right )}{\|v\|^2_{L^2(\Omega)}}$ which are available when the Allen-Cahn equation describes a smooth evolution of a developing interface.

Based on the above idea, for the numerical analysis of the implicit Euler scheme, in \cite{FePr04},  the first a-priori bounds were established in various norms with constants that depend upon $(1/\epsilon)$ in a polynomial fashion. For instance, for the energy norm an estimate of order $\tau +h$ with constant depending upon $1/\epsilon^3$, when the data $\|\nabla u_0\|_{L^2(\Omega)},  \|\Delta u_0\|_{L^2(\Omega)}, \lim_{s \to 0^{+}} \|\nabla u_t(s)\|_{L^2(\Omega)} \leq C$  and the spacial and the temporal discretization parameters satisfy $\tau + h^2 \leq C\epsilon^7$ and $h | \ln h |^{1/2} \leq \epsilon^3$ when $d=2$ and $\tau + h^2\leq \epsilon^{13}$, and $h  \leq \epsilon^6$, when $d=3$ repsectively. The idea of using the principle eigenvalue operator, was further used in order to obtain a-posteriori bounds in \cite{KeNoSc04}, and \cite{FeWu05}, while various a-posteriori estimates based on a discretized  version of the principle eigenvalue operator where obtained in the works of \cite{BaMu11}, \cite{BaMuOr11}, and \cite{GeMa14}. In \cite{ShYa10}, semi-implicit schemes of first order were studied, and conditional stability estimates were presented for semi-discrete (in time) approximations. In addition, a second order semi-implicit, semi-discrete in time scheme which is conditionally stable was also considered, in \cite{ShYa10}. Finally, extensive numerical studies of various numerical schemes for the Allen-Cahn equation are presented in \cite{JuZhZhDu15,ZhDu09}. For various results regarding discountinuous time-stepping schemes for the semi-linear parabolic PDEs, we refer the reader to \cite{ErJo95b,EsLa93,EsLaWi00,Th97}.

\subsection{Main Results}
Our main goal is to provide a rigorous stability analysis of a general class of
fully-discrete schemes and to prove best approximation a-priori error estimates.  The schemes considered here are discontinuous (in time) and
conforming in space. The motivation for using the discontinuous
(in time) Galerkin approach relies in its robust performance in a vast area of
problems whose solutions satisfy low regularity properties. 

The key feature of the discontinuous time stepping Galerkin schemes is
their ability to mimic the stability properties of the
corresponding continuous system. Indeed, we prove that the fully-discrete solution, computed by using discontinuous Galerkin (in time) and conforming finite elements in space of arbitrary order (in time and space), (denoted by $u_h$) satisfies the following unconditional stability estimates:
$$ \|u_h\|_{L^2[0,T;L^2(\Omega)]} \leq C, \quad\mbox{and} \quad \|u_h\|_{L^{\infty}[0,T;L^2(\Omega)]} + \|u_h\|_{L^2[0,T;H^1(\Omega)]} \leq \frac{C}{\epsilon},$$
where $C$ denotes a constant depending on the domain $\Omega$, the norms of $\|u_0\|_{L^2(\Omega)}$ and $\|f\|_{L^2[0,T;H^{-1}(\Omega)]}$ and the polynomial degree in time, but it is independent of $\tau,h,\epsilon$. 

In addition, using the above estimates, we are able to prove the following best approximation error estimate, 
$$\| \mbox{error} \|_{X} \leq \frac{C}{\epsilon^3}  ( \|u\|_{L^{\infty}[0,T;H^1(\Omega)]} + \|u\|_{L^2[0,T;H^2(\Omega)]} ) \|\mbox{best approximation error}\|_{X},$$
where $X= L^{\infty}[0,T;L^2(\Omega)] \cap L^2[0,T;H^1(\Omega)]$, and $C$ denotes an algebraic constant depedning only upon data, and it is independent of $\tau,h,\epsilon$. For the above best approximation error estimate the spatial and temporal discretization parameters (denoted by $h$ and $\tau$)  satisfy $\tau +h \leq \frac{C\epsilon^{7/2}}{\|u\|_{L^2[0,T;H^2(\Omega)]} + \|u_t\|_{L^2[0,T;H^1(\Omega)]}}$ when $d=2$, and $\tau+ h \leq \frac{C\epsilon^{4}}{\|u\|_{L^2[0,T;H^2(\Omega)]} + \|u_t\|_{L^2[0,T;H^1(\Omega)]}}$ when $d=3$. 

The above estimate states that the error is as good as the approximation properties of the underlying subspaces, and the regularity of the solution will allow it to be. Therefore, it can be viewed as a generalization of the the classical Cea's Lemma. 

To our best knowledge, so far, in the literature, there has been no rigorous proof regarding unconditional stability as well as best-approximation type of error estimates for any kind of fully-discrete scheme with polynomial dependence on the the quantity $(1/\epsilon)$. The scope of this work, is to prove that for a very broad category of fully-discrete schemes such (unconditional) stability estimates as well as best approximation error estimates (in the spirit of the classical Cea's Lemma) are possible, even under low regularity assumptions on the given data. 

\subsection{Our approach}
We close our introduction, by introducing the main idea which is essential for the analysis of our estimates. For the stablity analysis, instead of focusing on the uniform bounds of the principle eigenvalue of the linearized (elliptic) part of the  Allen-Cahn operator, we define the following auxiliary (almost dual) linearized pde, with appropriate scaling in $L^2[0,T;L^2(\Omega)]$ norm. In particular, with right hand side $u \in L^2[0,T;L^2(\Omega)]$, and zero terminal data $\phi(T) =0$, we seek $\phi \in L^2[0,T;H^1(\Omega)] \cap L^{\infty}[0,T;L^2(\Omega)]$ satisfying
\begin{equation*}
\ - \phi_t - \Delta \phi + \frac{1}{\epsilon^2} u^2 \phi + \frac{1}{\epsilon^2}  \phi = u, \quad\mbox{in $(0,T) \times \Omega$}, \qquad \phi = 0 \quad\mbox{ on $(0,T) \times \Gamma$}
\end{equation*}

The key ingredient in our stability analysis is the construction of the fully-discrete space-time approximation of the above linearized Allen-Cahn equation with an appropriately scaled $L^2[0,T;L^2(\Omega)]$ part, based on the discontinuous time-stepping Galerkin formulation. This auxialiary space-time projection effectively allows to apply  a duality argument, to recover first the unconditional stability with respect to $L^2[0,T;L^2(\Omega)]$ norm, and then a boot-strap argument to recover the unconditional  stability in $L^2[0,T;H^1(\Omega)]$ and $L^{\infty}[0,T;L^2(\Omega)]$. For the later we employ the techniques developed by \cite{ChWa06,ChWa10,Wa10}, in a way to avoid the use of a Gr\"onwall's type argument. The discrete compactness argument of Walkington (see \cite{Wa10}), then allows to rigorously pass the limit to prove convergence. We note that the case of zero Neumann boundary data can be also considered in an identical way. The use of parabolic duality was initiated in the paper of \cite{LuRa82}, for the derivation of semi-discrete in space estimates for general linear parabolic PDEs, using the smoothing property (see also \cite[Chapter 12]{Th97} and references within for related results in the context of discontinuous time-stepping methods). 

For the best approximation error estimate we employ a similar strategy and the stability estimates in crucial way. To seperate the difficulties due to the nonlinear structure from the ones involving the different scaling (in terms of $\epsilon$) of various norms, we derive estimates in three steps:
\begin{enumerate}
\item We define an auxiliary space-time (linear) parabolic projection that exhibits  best approximation error estimates. The auxiliarly space-time parabolic projection $u_p$ is defined as the discnotinuous time stepping solution of a linear parabolic pde with right hand side $u_t- \Delta u$, and appropriate initial data, and using the result of \cite[Section 2]{ChWa06} and a proper duality argument we obtain best approximation estimates for the difference between $u-u_p$.
\item We use a duality argument, combined with the previously developed stability estimates to obtain the key preliminary estimate for the $L^2[0,T;L^2(\Omega)]$ norm without using Gr\"onwall's type arguments, with constants depending polynomially upon $\frac{1}{\epsilon}$. To achieve this, first we employ the discrete compactness argument of Walkington \cite{Wa10} to recover strong convergence in $L^4[0,T;L^2(\Omega)]$ to guarantee that the error $u_h-u$ is small enough (for small enough discretization parameters $\tau,h$). Then, we define the space-time discontinuous Galerkin approximation of the problem,
\begin{eqnarray*}
&& - \psi_t - \Delta \psi + \frac{1}{\epsilon^2} (3u^2-1) \psi = u_h-u_p, \quad \psi(T) = 0, \quad \psi|_{(0,T) \times \Gamma} = 0 
\end{eqnarray*}
and we prove various key stability estimates, with the help of the spectral estimate. 
\item  Then, we recover the full rate in the $L^2[0,T;H^1(\Omega)]$ norm via a boot-strap argument and the estimate at arbitrary time-points via the techniques developed by \cite{ChWa06,ChWa10,Wa10} to obtain the symmtric structure of the best-approximation error estimate. The boot-stap argument is performed in a way to avoid the use of a Gr\"onwall type arguments.
\end{enumerate}
The remaining of the paper is organized as follows: In section 2, we present the necessary notation, and some preliminary estimates for weak solutions of the Allen-Cahn equation. In Section 3, after defining the fully-discrete discontinuous Galerkin scheme, we present the basic stability estimates, which allows us to establish unconditional estimates in $L^{\infty}[0,T;L^2(\Omega)]$ and to prove strong convergence in Section 4. Finally in Section 5, we prove best-approxation estimates with constants depending polynomially upon $\frac{1}{\epsilon}$ and apply these results to obtain convergence rates.

\section{Preliminaries}
\subsection{Notation}
Let $U$ denote a Banach space. Typically, $U
\equiv H^s(\Omega), 0<s \in \mathbb R$, where $H^s(\Omega)$
denotes the standard Sobolev (Hilbert) spaces (see for instance \cite{E,Ze}). We denote by $H^0(\Omega) \equiv L^2(\Omega)$, and by $H^1_0(\Omega)
\equiv \{ w \in H^1(\Omega) : w|_{\Gamma} =0 \}$. Finally, we use the notation $\langle
.,. \rangle$ for the duality pairing of
$H^{-1}(\Omega),H^1_0(\Omega)$ and $(.,.)$ for the standard $L^2$
inner product, where $H^{-1}(\Omega)$ is the dual space of $H^1_0(\Omega)$.
We denote the time-space spaces by
$L^p[0,T;U], L^{\infty}[0,T;U]$, endowed with norms:
$$\|w\|_{L^p[0,T;U]} = \Big ( \int_0^T \|w\|^p_{U} dt \Big )^{\frac{1}{p}},
\quad \|w\|_{L^{\infty}[0,T;U]} = \mbox{ ess$ \sup_{t \in [0,T]}
\|w\|_{U}$}.$$ The set of all continuous functions $v : [0,T]
\rightarrow U$, is denoted by $C[0,T;U]$ with norm
$ \|w\|_{C[0,T;U]} = \max_{t \in [0,T]} \|w(t)\|_{U}. $
For the definition of spaces
$H^s[0,T;U]$, we refer the reader to \cite{E,Ze}.  Throughout this work we will use the following
(natural energy) space for the solution $u$ of (\ref{eqn:ac}),
$$ X = L^{\infty}[0,T;L^2(\Omega)] \cap L^2[0,T;H^1_0(\Omega)] $$
with associated norm
$\|w\|^2_{X} = \|w\|^2_{L^{\infty}[0,T;L^2((\Omega)]} +
\|w\|^2_{L^2[0,T;H^1(\Omega)]}.$
The bilinear form related to our problem is defined by
$$ a(w_1,w_2) = \int_\Omega \nabla w_1 \nabla w_2 dx \qquad\forall\, w_1,w_2 \in
H^1_0(\Omega). $$ Using Poincar\'e's inequality we obtain the
corresponding coercivity condition
$$ a(w,w) \geq C \|w\|^2_{H^1(\Omega)} \qquad\forall\, w \in
H^1_0(\Omega),$$ where $C$ denotes an algebraic constant depending only upon the domain $\Omega$. We close this preliminary section, by recalling
Young's inequality and
Landyzeskaya-Gagliardo-Nirenberg interpolation inequalities. \\
{\it Young's Inequality:} For any $a,b\geq 0$ any $\delta
>0$, and $s_1,s_2>1$
\begin{equation*} \label{eqn:young}
ab \leq \delta a^{s_1} + C(s_1,s_2) \delta^{-\frac{s_2}{s_1}} b^{s_2}, \qquad\mbox{where
$(1/s_1)+(1/s_2)=1$}.
\end{equation*}
{\it Landyzeshkayka-Gagliardo-Nirenberg Interpolation
 Inequalities:} There exist constant $C>0$ depending only upon the domain such that, for all $u \in H^1_0(\Omega)$,
\begin{eqnarray*}
&& \|u\|_{L^4(\Omega)} \leq C\|u\|^{1/2}_{L^2(\Omega)}
\|u\|^{1/2}_{H^1(\Omega)}, \quad\mbox{ when $d=2$}, \\
&& \|u\|_{L^3(\Omega)} \leq C\|u\|^{1/2}_{L^2(\Omega)} \|u\|^{1/2}_{H^1(\Omega)}, \quad\mbox{ when $d=3$}, \\
&& \|u\|_{L^4(\Omega)} \leq C\|u\|^{1/4}_{L^2(\Omega)} \|u\|^{3/4}_{H^1(\Omega)}, \quad\mbox{ when $d=3$}.
\end{eqnarray*}
\subsection{Weak formulation and regularity of the Allen-Cahn equation}
The following weak formulation of
(\ref{eqn:ac}) will be used subsequently.
Let $f \in L^2[0,T;H^{-1}(\Omega)]$ and $u_0
\in L^2(\Omega)$. Then, for all $w \in H^1_0(\Omega)$ and for a.e.
$t \in (0,T]$, we seek $u \in L^2[0,T;H^1_0(\Omega)] \cap H^1[0,T;H^{-1}(\Omega)]$ such that
\begin{equation*} \label{eqn:wac} \langle u_t,w \rangle + a(u,w) + (1/\epsilon^2) \langle
u^3-u,w \rangle = \langle f,w \rangle, \quad\mbox{and}\quad
(u(0),w) = (u_0,w).
\end{equation*}
Since, our schemes are based on the discontinuous time-stepping framework, a suitable weak formulation can be written as follows:
We seek $u \in L^{\infty}[0,T;L^2(\Omega)] \cap L^2[0,T;H^1_0(\Omega)]$,
satisfying,
\begin{eqnarray} \label{eqn:wac2}
&& (u(T),w(T)) + \int_0^T \Big ( -\langle u,w_t \rangle  + a(u,w)
+ \frac{1}{\epsilon^2} (u^3-u,w)  \Big )dt \nonumber \\
&& = (u_0,w(0)) + \int_0^T \langle f,w \rangle dt \,
\end{eqnarray}
for all $w \in L^2(0,T;H^1_0(\Omega)) \cap H^1(0,T;H^{-1}(\Omega))$. It clear that using straightforward techniques (see for instance \cite{Te,Ze}) one can easily prove the existence a weak solution  solution $u \in L^{\infty}[0,T;L^2(\Omega)] \cap L^2[0,T;H^1_0(\Omega)]$ which satisfies
the following energy estimate
\begin{equation*}
\|u\|_{X}  \leq C_\epsilon \Big (
\|f\|_{L^2[0,T;H^{-1}(\Omega)]} + \|u_0\|_{L^2(\Omega)} \Big ),
\end{equation*}
where $C_\epsilon$ depends on $\Omega$, and the parameters
$\epsilon$ and $T$. 

The following Lemma quantifies the dependence upon $\epsilon$ of various norms.
\begin{lem} \label{lem:contreg}
Suppose that $f \in L^2[0,T;H^{-1}(\Omega)]$ and $u_0 \in L^2(\Omega)$. Then, there exists a constant C (independent of $\epsilon$) such that:
\begin{align*}
&\|u\|_{L^2[0,T;L^2(\Omega)]} + \|u\|^2_{L^4[0,T;L^4(\Omega)]} \leq C \left ( T^{1/2} + \epsilon (\|u_0\|_{L^2(\Omega)} + \|f\|_{L^2[0,T;H^{-1}(\Omega)]}) \right ), \\
& \|u\|_{L^{\infty}[0,T;L^2(\Omega)]} + \|u\|_{L^2[0,T;H^1(\Omega)]} \leq \frac{C}{\epsilon}.
\end{align*}
Suppose that 
\begin{equation} \label{eqn:regass}
f \in L^2[0,T;L^2(\Omega)] \quad \mbox{ and}\quad \|\nabla u_0\|_{L^2(\Omega)} + \frac{1}{\epsilon^2} \| (1/4) (u^2_0-1)^2\|_{L^1(\Omega)} \leq C.
\end{equation} 
Then, there exists a constant $C$ (indpendent of $\epsilon$) such that the following estimate holds:
\begin{equation} \label{eqn:contstab}
\|u\|_{L^2[0,T;H^2(\Omega)]} \leq \frac{C}{\epsilon}, \qquad \|u\|_{L^{\infty}[0,T;H^1(\Omega)]} + \|u_t\|_{L^2[0,T;L^2(\Omega)]} \leq C.
\end{equation}
\end{lem}
\begin{proof}
For the first estimate, we use the following auxiliary backward in time linear parabolic pde. Let $u$ be the solution of (\ref{eqn:wac2}). Given, right hand side $u \in L^2[0,T;L^2(\Omega)]$ and terminal data $\phi (T) =0$, we seek $\phi \in L^2[0,T;H^1_0(\Omega)] \cap H^1[0,T;H^{-1}(\Omega)]$ such that, for all $w \in L^2[0,T;H^1_0(\Omega)] \cap H^1[0,T;H^{-1}(\Omega)]$, 
\begin{equation} \label{eqn:l1}
 \int_0^T \left  ( (\phi, w_t) + a(\phi,w) + \frac{1}{\epsilon^2} (u^2 \phi,w) + \frac{1}{\epsilon^2} (\phi,w) \right ) dt + (\phi(0),w(0)) = \int_0^T (u,w) dt.
\end{equation}
It is clear that setting $w= \phi$ in (\ref{eqn:l1}) we obtain the following bound:
\begin{align} \label{enq:l2} 
& \frac{1}{2} \|\phi(0)\|_{L^2(\Omega)]} + C\|\phi\|_{L^2[0,T;H^1(\Omega)]} + \frac{1}{\epsilon} \|\phi u\|_{L^2[0,T;L^2(\Omega)]} + \frac{1}{2 \epsilon} 
\|\phi\|_{L^2[0,T;L^2(\Omega)]} \nonumber \\
& \leq \frac{\epsilon}{2} \|u\|_{L^2[0,T;L^2(\Omega)]}. 
\end{align}
Now, we employ a ``duality'' argument. Integrating by parts in time (\ref{eqn:wac2}), and setting $w = \phi$ into the resulting equation, we obtain:
\begin{equation} \label{eqn:l3}
  \int_0^T \Big ( \langle u_t,\phi \rangle  + a(u,\phi)
+ \frac{1}{\epsilon^2} (u^3-u,\phi)  \Big )dt = \int_0^T \langle f,\phi \rangle dt. 
\end{equation}
Setting $w= u$ into (\ref{eqn:l1}) and subtracting the resulting equality from (\ref{eqn:l3}) we derive:
\begin{equation} \label{eqn:l4}
\int_0^T \|u\|^2_{L^2(\Omega)} dt = \frac{2}{\epsilon^2} \int_0^T (\phi,u) dt + \int_0^T \langle f,\phi \rangle dt + (\phi(0),u(0)).
\end{equation}
Note that using H\"older's inequality, and the stability estimates, equation (\ref{eqn:l4}) implies that
\begin{eqnarray*} 
&& \|u\|^2_{L^2[0,T;L^2(\Omega)]} \leq \frac{2}{\epsilon^2} \int_0^T |\Omega|^{1/2} \|\phi u\|_{L^2(\Omega)} dt \\
&& \qquad + \|f\|_{L^2[0,T;H^{-1}(\Omega)} \|\phi\|_{L^2[0,T;H^1(\Omega)]} + \|\phi(0)\|_{L^2(\Omega)} \|u(0)\|_{L^2(\Omega)} \\
&& \leq \frac{2}{\epsilon^2} |\Omega|^{1/2}  T^{1/2} \|\phi u\|_{L^2[0,T;L^2(\Omega)]} \\
&& \qquad+ C( \|f\|_{L^2[0,T;H^{-1}(\Omega)]} + \|u(0)\|_{L^2(\Omega)} \big ) \epsilon \|u\|_{L^2[0,T;L^2(\Omega)]} \\
&& \leq \frac{2}{\epsilon^2} |\Omega|^{1/2} T^{1/2} \frac{\epsilon^2}{2} \|u\|_{L^2[0,T;L^2(\Omega)]} \\
&&\qquad +C( \|f\|_{L^2[0,T;H^{-1}(\Omega)]} + \|u(0)\|_{L^2(\Omega)} \big ) \epsilon \|u\|_{L^2[0,T;L^2(\Omega)]},
\end{eqnarray*}
which implies the desired estimate on $\|u\|_{L^2[0,T;L^2(\Omega)]}$. Returning back to (\ref{eqn:wac2}), setting $w=u$, and using the bound on $\|u\|_{L^2[0,T;L^2(\Omega)]}$ we obtain the first estimate. 
For the second estimate, we set $w = u_t$, and we observe,
\begin{equation*}
\int_0^T \left ( \|u_t\|^2_{L^2(\Omega)} + \frac{d}{dt} \left ( \frac{1}{2} \|\nabla u\|^2_{L^2(\Omega)} + \frac{1}{4 \epsilon^2} \|(u^2-1)^2\|_{L^1(\Omega)} \right ) \right ) dt
= \int_0^T  (f,u_t) dt.
\end{equation*}
The estimate now follows by standard algebra. The estimate on $\|\Delta u\|_{L^2[0,T;L^2(\Omega)]}$ follows using standard techniques.
 \end{proof}
\begin{rmk}
If more regularity is available, then we can quantify the dependence upon $1/\epsilon$ in other norms (see for instance \cite[Proposition 1]{FePr04}). In addition to (\ref{eqn:regass}), if the initial data satisfy,
$\|\Delta u_0 \|_{L^2(\Omega)} \leq C$ 
with constant $C$ independent of $\epsilon$, then, 
$$
\|u\|_{L^{\infty}[0,T;H^2(\Omega)]} + \|u_t\|_{L^{\infty}[0,T;L^2(\Omega)]} \leq \frac{C}{\epsilon}, \qquad \|\nabla u_t\|_{L^2[0,T;L^2(\Omega)]} \leq \frac{C}{\epsilon}.
$$
We point out that the regularity bound on $ \frac{1}{\epsilon^2} \| (1/4) (u^2_0-1)^2\|_{L^1(\Omega)} \leq C$ is essential in order to obtain (\ref{eqn:contstab}). It is worth noting that if only $\|u_0\|_{H^1(\Omega)} \leq C$ is assumed then the dependence upon $\frac{1}{\epsilon}$ deteriorates to:
$$ \|u\|_{L^{\infty}[0,T;H^1(\Omega)]} + \|u_t\|_{L^{2}[0,T;L^2(\Omega)]} + \|u\|_{L^2[0,T;H^2(\Omega)]} \leq \frac{C}{\epsilon^2}. $$

\end{rmk}
For the stability analysis of the fully-discrete schemes, enhanced regularity assumptions, such as $u \in L^{\infty}[0,T;H^2(\Omega)] \cap H^1[0,T;H^1(\Omega)]$  are not necessary. For the error estimates, the constants will depend upon the norms of  
$\|u\|_{L^{\infty}[0,T;H^1(\Omega)]}$, $\|u_t\|_{L^2[0,T;L^2(\Omega)]}$ and $\|u\|_{L^2[0,T;H^2(\Omega)]}$.
\section{The fully-discrete scheme}
\setcounter{equation}{0}
\subsection{The discontinuous time-stepping approximations}
For the discretization of the Allen-Cahn model we employ a discontinuous time-stepping Galerkin approach, combined with standard conforming finite element subspaces.
Approximations will be constructed on a partition $0=t^0 < t^1 <
\ldots < t^N=T$ of $[0,T]$. On each interval of the form
$(t^{n-1},t^n]$ of length $\tau_n=t^n-t^{n-1}$, a subspace $U_h$
of $H^1_0(\Omega)$ is specified for all $n=1,..,N$ and it is assumed that each
$U_h$ satisfies the classical approximation theory results (see e.g. \cite{Ci}), on regular
meshes. In particular, we assume that there exists an integer
$\ell \geq 1$ and a constant $c>0$ (independent of the mesh-size parameter $h$) such
that if $w \in H^{l+1}(\Omega) \cap H^1_0(\Omega)$,
$$ \inf_{w_h \in U_h} \|w-w_h\|_{H^s(\Omega)} \leq Ch^{l+1-s} \|w\|_{H^{l+1}(\Omega)},
\qquad 0 \leq l \leq \ell, \qquad s=-1,0,1.
$$
We also assume that the partition is quasi-uniform in time, i.e., there exists
a constant $0<\theta \leq 1$ such that $\theta \tau \leq \min_{n=1,...N} \tau_n$,
where $\tau = \max_{n=1,...,N} \tau_n$. We seek approximate solutions
which belong to the space
$$
{\mathcal U}_h = \{w_h \in L^2[0,T;H^1_0(\Omega)] :
        w_h|_{(t^{n-1},t^n]} \in {\mathcal P}_k[t^{n-1},t^n; U_h] \}.
$$
Here ${\mathcal P}_k[t^{n-1},t^n; U_h]$ denotes the space of
polynomials of degree $k$ or less having values in $U^n_h$.  By
convention, the functions of ${\mathcal U}_h$ are left continuous
with right limits and hence we will subsequently write $w^n_{h-}$ for
$w_h(t^n) = w_h(t^n_-)$, and $w^n_{h+}$ for $w_h(t^n_+)$. Note that, we
have also used the following notational abbreviation, $w_h \equiv w_{h,\tau}$,
${\mathcal U}_h \equiv {\mathcal U}_{h,\tau}$ etc, since for the stability analysis we will not impose any restriction involving $\tau$, and $h$. The jump at $t^{n}$ will be denoted as 
$[w^{n}_h] = w^{n}_{h+}-w^{n}_{h-}$. 
The fully discrete system is defined as follows: We seek $u_h
\in {\mathcal U}_h$ such that for every $w_h \in {\mathcal U}_h$ and for $n=1,...,N$,
\begin{eqnarray}
&&(u^n_{h-},w^n_{h-}) + \int_{t^{n-1}}^{t^n}
        \Big(- \langle u_h,w_{ht} \rangle +  a(u_h,w_h) +
        (1/\epsilon^2) (u^3_h-u_h,w_h) \Big)dt
        \label{eqn:dac} \nonumber \\
&&= (u^{n-1}_{h-},w^{n-1}_{h+})
        + \int_{t^{n-1}}^{t^n} \langle f,w_h \rangle dt.
\end{eqnarray}

Recall  that $f,u_0$ are given data, and
$u^0$ denotes approximations of $u_0$. In our case, we will define $u^0 = P_hu^0$, where $P_h$ denotes the standard $L^2$ projection, i.e.,
$P_h: L^2(\Omega) \to U_h$, $(P_hv-v,w_h)=0, \quad\forall\, w_h \in U_h$.

\begin{rmk} \label{rmk:exist}
For any $\epsilon >0$, existence and uniqueness of discontinuous Galerkin approximations of (\ref{eqn:dac}) can be proved easily (even for more complicated nonlinearities) due to finite dimensionality of the problem. For several results regarding discontinuous time-stepping schemes, with linear and semi-linear terms, we refer the reader to the works
\cite{AkMa04,DeHaTr81,ErJo91,ErJo95,ErJo95b,Ja78,MeVe08,Th97,Wa10} (see also references within).
\end{rmk}

\subsection{The basic estimate using duality}
We begin by developing a stability estimate via duality for the $L^2[0,T;L^2(\Omega)]$ norm.
For this purpose, we define a backward in time parabolic problem with right hand side $u_{h} \in L^2[0,T;L^2(\Omega)]$ with an enhanced $L^2[0,T;L^2(\Omega]$ term and zero terminal data. In particular, for right hand side $u_h \in L^2[0,T;L^2(\Omega)]$, and terminal data $\phi^N_{h+}=0$, we seek $\phi_h \in {\mathcal U}_h$ such that for all $w_h \in {\mathcal P}_k[t^{n-1},t^n;U_h]$, and for $n=N,...,1$, 
\begin{eqnarray} \label{eqn:dual}
&& -({\phi}^n_{h+},w^n_{h-}) + \int_{t^{n-1}}^{t^n} \big ( (\phi_h,w_{ht})
+ a({\phi}_h,w_h) + (1/\epsilon^2) \langle u^2_h \phi_h,w_h \rangle  \big ) \nonumber \\
&& + \int_{t^{n-1}}^{t^n}  (1 / \epsilon^2) (\phi_h,w_h)dt
 +({\phi}^{n-1}_{h+},w^{n-1}_{h+}) = \int_{t^{n-1}}^{t^n} ({u}_{h},w_h)dt.
\end{eqnarray}
Note that is easy to prove existence at partition points as well as in $L^2[0,T;H^1_0(\Omega)]$, due to the signs of the inner products $(1/\epsilon^2) (u^2_h \phi_h,w_h )$ and $(1 / \epsilon^2) (\phi_h,w_h)$. Given, $u_h \in {\mathcal U}_h$, it is obvious that $\phi_h \in {\mathcal U}_h$ is unique. In Section 4.2, we will also prove that $u_h \in L^{\infty}[0,T;L^2(\Omega)]$.
\begin{lem} \label{lem:stab}
Let $f \in L^2[0,T;H^{-1}(\Omega)]$, $u_0 \in L^2(\Omega)$, and $u_h \in {\mathcal U}_h$ are the solutions of (\ref{eqn:dac})-(\ref{eqn:dual}) respectively. Then, there exists a constant $C>0$, depending only upon the domain $\Omega$, $T$, and which is independent of $\epsilon$ such that:
\begin{equation*}
\|u_h\|_{L^2[0,T;L^2(\Omega)]} \leq C \left ( T^{1/2} + \epsilon (\|u_0\|_{L^2(\Omega)} + \|f\|_{L^2[0,T;H^{-1}(\Omega)]}) \right )
\end{equation*}
In addition, the following estimates hold: For all $n=1,...,N$
\begin{eqnarray*}
&& \|u^n_{h-}\|_{L^2(\Omega)} + \|u_h\|_{L^2[0,T;H^1(\Omega)]} + (1/\epsilon) \|u_h\|^2_{L^4[0,T;L^4(\Omega)]} + \sum_{i=1}^N \|[u^i_h]\|^2_{L^2(\Omega)}  \\
&& \leq (C/\epsilon) \left ( \|u_0\|_{L^2(\Omega)} + \|f\|_{L^2[0,T;H^{-1}(\Omega)]} \right ).
\end{eqnarray*}
where $C$ is a constant depending only upon $\Omega, T$.
\end{lem}
\begin{proof}
 Setting $w_h = \phi_h$, into (\ref{eqn:dual}), using Young's inequality to bound $$\int_{t^{n-1}}^{t^n} ({u}_{h},\phi_h)dt \leq (1/2\epsilon^2) \int_{t^{n-1}}^{t^n} \|\phi_h\|^2_{L^2(\Omega)} + (\epsilon^2 /2) \int_{t^{n-1}}^{t^n} \|u_h\|^2_{L^2(\Omega)} dt,$$ and adding the resulting terms, we derive the following estimate. For all $n=N,...,1$
\begin{eqnarray}
\label{eqn:dual1} && \|\phi^{n-1}_{h+}\|^2_{L^2(\Omega)} +  \|\nabla \phi_h\|^2_{L^2[0,T;L^2(\Omega)]} + (1/ \epsilon^2) \|\phi_h u_h\|^2_{L^2[0,T;L^2(\Omega)]}  \nonumber \\
&& \qquad + (1 / 2 \epsilon^2) \|\phi_h\|^2_{L^2[0,T;L^2(\Omega)]} \leq (\epsilon^2 / 2) \|u_{h}\|^2_{L^2[0,T;L^2(\Omega)]}.
\end{eqnarray}
Now setting $w_h = u_h$ into (\ref{eqn:dual}), we easily derive
\begin{eqnarray*} 
&&\hskip-20pt -({\phi}^n_{h+},u^n_{h-}) + \int_{t^{n-1}}^{t^n} \left ( ({\phi}_h,u_{ht}) +
a(u_{h},\phi_h) + (1/\epsilon^2) \langle u^2_h \phi_h, u_{h} \rangle + (1/\epsilon^2) (\phi_h,u_h) \right ) dt \\
&&\hskip-20pt+ ({\phi}^{n-1}_{h+},u^{n-1}_{h+})  = \int_{t^{n-1}}^{t^n} \|u_{h}\|^2_{L^2(\Omega)} dt.
\end{eqnarray*}
Integrating by parts in time, we deduce,
\begin{eqnarray} \label{eqn:dual2}
&& -({\phi}^n_{h+},{u}^n_{h-}) + ({\phi}^n_{h-},{u}^n_{h-}) +
\int_{t^{n-1}}^{t^n} \left ( - \langle {\phi}_{ht},{u}_{h} \rangle + a({\phi}_h,{u}_{h}) \right ) dt \nonumber \\
&& + \int_{t^{n-1}}^{t^n} \left ( (1/\epsilon^2) \langle u^2_h \phi_h, u_{h} \rangle + (1/ \epsilon^2) (\phi_h,u_h) \right ) dt = \int_{t^{n-1}}^{t^n} \|u_{h}\|^2_{L^2(\Omega)} dt.
\end{eqnarray}
Setting $w_h = {\phi}_h$ into (\ref{eqn:dac}), we obtain,
\begin{eqnarray}
&& (u^n_{h-},\phi^n_{h-})+ \int_{t^{n-1}}^{t^n}
        \left ( - \langle u_h,\phi_{ht} \rangle +  a(u_h,\phi_h) +
        (1/\epsilon^2) \langle u^3_h-u_h,\phi_h \rangle \right )dt
        \label{eqn:dual3} \nonumber \\
&& = (u^{n-1}_{h-},\phi^{n-1}_{h+})
        + \int_{t^{n-1}}^{t^n} \langle f,\phi_h \rangle dt.
\end{eqnarray}
Subtracting (\ref{eqn:dual3}) from (\ref{eqn:dual2}), and noting that the terms $(1/\epsilon^2) \int_{t^{n-1}}^{t^n} \int_\Omega u^3_h \phi_h dxdt$ are canceled,  we arrive to
\begin{eqnarray} \label{eqn:dual4}
&& ({\phi}^n_{h+},{u}^n_{h-}) - ({u}^{n-1}_{h-},{\phi}^{n-1}_{h+}) +
\int_{t^{n-1}}^{t^n}  \|{u}_{h}\|^2_{L^2(\Omega)} dt  \nonumber \\
&& = (1/ \epsilon^2) \int_{t^{n-1}}^{t^n} (\phi_h,u_h) dt - \int_{t^{n-1}}^{t^n} \langle f, \phi_h \rangle dt + (1/\epsilon^2) \int_{t^{n-1}}^{t^n} (u_h,\phi_h) dt.
\end{eqnarray}
First, we treat the terms involving $(1/\epsilon^2)$ constants. Using Young's inequality with appropriate $\delta_1>0$ (to be determined later), we deduce,
\begin{eqnarray*}
&& (2/\epsilon^2) \int_{t^{n-1}}^{t^n} \left | (\phi_h,u_h) \right | dt \leq (2/\epsilon^2) \int_{t^{n-1}}^{t^n} |\Omega|^{1/2} \|\phi_h u_h\|_{L^2(\Omega)} dt \\
&& \leq (2/\epsilon^2) \tau^{1/2}_n |\Omega|^{1/2} \left (\int_{t^{n-1}}^{t^n} \|\phi_h u_h\|^2_{L^2(\Omega)} dt \right)^{1/2} \\
&& \leq (2\delta_1 / \epsilon^2) \tau_n |\Omega| + (1 /2 \delta_1 \epsilon^2) \int_{t^{n-1}}^{t^n} \|\phi_h u_h\|^2_{L^2(\Omega)} dt.
\end{eqnarray*}
Similarly, using Young's inequality with appropriate $\delta_2>0$, we obtain
$$\int_{t^{n-1}}^{t^n} | \langle f, \phi_h \rangle | dt \leq (\delta_2/\epsilon^2) \int_{t^{n-1}}^{t^n} \|\phi_h\|^2_{H^1(\Omega)} + (\epsilon^2 / 4 \delta_2 ) \int_{t^{n-1}}^{t^n} \|f\|^2_{H^{-1}(\Omega)} dt.$$
Substituting the last two inequalities into (\ref{eqn:dual4}), summing the resulting inequalities and using the fact that
${\phi}^N_{+}\equiv 0$ (by definition) and rearranging terms, we obtain
\begin{eqnarray*}
&& \|u_{h}\|^2_{L^2[0,T;L^2(\Omega)]} \leq \|u^0_h\|_{L^2(\Omega)} \|\phi^0_{h+}\|_{L^2(\Omega)} + (\delta_2 / \epsilon^2) \|\phi_h\|^2_{L^2[0,T;H^1(\Omega)]} \\
&&  + ( \epsilon^2 /4\delta_2) \|f\|^2_{L^2[0,T;H^{-1}(\Omega)]} + (2 \delta_1 / \epsilon^2) \sum_{n=1}^N \tau_n |\Omega| + (1/2 \delta_1 \epsilon^2) \|\phi_h u_h\|^2_{L^2[0,T;L^2(\Omega)]} \\
&& \leq (\delta_3 / \epsilon^2) \|\phi^0_{h+}\|^2_{L^2(\Omega)} + (\epsilon^2 / 4 \delta_3) \|u^0_h\|^2_{L^2(\Omega)}  + (\delta_2/\epsilon^2) \|\phi_h\|^2_{L^2[0,T;H^1(\Omega)]} \\
&& + (\epsilon^2 /4\delta_2) \|f\|^2_{L^2[0,T;H^{-1}(\Omega)]} + (2 \delta_1 / \epsilon^2) \sum_{n=1}^N \tau_n |\Omega| + (1/2 \delta_1 \epsilon^2) \|\phi_h u_h\|^2_{L^2[0,T;L^2(\Omega)]}.
 \end{eqnarray*}
 Using the previous bounds on $\|\phi^0_{h+}\|_{L^2(\Omega)}$, $\|\phi_h\|_{L^2[0,T;H^1(\Omega)]}$, $(1/\epsilon) \|\phi_h\|_{L^2[0,T;L^2(\Omega)]} $, and $(1/\epsilon) \|\phi_h u_h\|_{L^2[0,T;L^2(\Omega)]}$, in terms of $\|u_h\|_{L^2[0,T;L^2(\Omega)]}$ via (\ref{eqn:dual1}) and choosing $\delta_1 = 2\epsilon^2$, $\delta_2 = \delta_3 = 1/4$, to hide the resulting terms on the left, we obtain,
 \begin{eqnarray*}
 \|u_h\|_{L^2(0,T;L^2(\Omega)]} \leq C \left ( T^{1/2} + \epsilon \left ( \|u^0_h\|_{L^2(\Omega)} + \|f\|_{L^2[0,T;H^{-1}(\Omega)]} \right )  \right ).
 \end{eqnarray*}
with $C$ an algebraic constant, depending only upon $| \Omega |$.
Setting $w_h = u_h$, in (\ref{eqn:dac}) respectively and using the
Poincar{\'e}, and Young's inequalities we obtain:
\begin{eqnarray} \label{eqn:dual6}
&& (1/2) \|u^n_{h-}\|^2_{L^2(\Omega)} - (1/2)
\|u^{n-1}_{h-}\|^2_{L^2(\Omega)} + (1/2) \| [u^{n-1}_h]\|^2_{L^2(\Omega)} \nonumber \\
&& + \int_{t^{n-1}}^{t^n} \Big (
  (C/2) \|u_h\|^2_{H^1(\Omega)} + (1/\epsilon^2) \|u_h\|^{4}_{L^{4}(\Omega)}
  \Big ) dt \nonumber \\
&&  \leq
(1/\epsilon^2) \int_{t^{n-1}}^{t^n}
 \|u_h\|^2_{L^2(\Omega)} dt
       + \int_{t^{n-1}}^{t^n}  (1 / C) \|f\|^2_{H^{-1}(\Omega)}
       dt.
\end{eqnarray}
The second estimate follows by summation and the previously developed estimate on $L^2[0,T;L^2(\Omega)]$.
\end{proof}
We close this subsection by a short remark.
\begin{rmk}
It is evident that the key estimate with respect the dependence upon $(1/\epsilon)$ concerns the term $(1/\epsilon^2) \int_{t^{n-1}}^{t^n} \int_\Omega u_h w_h dxdt$ which has the wrong sign and not the term $(1/\epsilon^2) \int_{t^{n-1}}^{t^n} \int_\Omega u^3_h w_h dxdt$ which is positive when setting $w_h = u_h$. For this reason the estimate of (\ref{eqn:dac}) does not lead to an estimate, with bounds independent of $\exp(1/\epsilon)$ when using Gronwall type arguments even for the lowest order scheme. To the contrary the duality argument of Lemma \ref{lem:stab}, leads to polynomial dependence upon $(1/\epsilon)$, without imposing any condition between $\tau,h$, and under minimal regularity assumptions. The key question regarding the stability at arbitrary time-points, i.e. in $L^{\infty}[0,T;L^2(\Omega)]$, will be considered next.
\end{rmk}

\section{Estimates at arbitrary time-points and convergence under minimal regularity}
We will employ the theory of the approximation of the discrete characteristic
functions (see e.g. \cite{ChWa06,ChWa10,Wa10}), which was used to develop estimates at arbitrary time points for a general class linear parabolic PDEs and for the Navier-Stokes respectively. The main advantage of this
approach is that the proof does not need any additional
regularity, apart from the one needed to guarantee the existence
of a weak solution, i.e., we do not assume that $u_{t} \in
L^2[0,T;L^2(\Omega)]$ which is frequently used in the literature
for dG approximations of parabolic PDEs. In addition, we will be able to obtain stability estimates without assuming any explicit dependence upon $\tau$ and $h$. A key feature of our analysis is that we are able to include high order schemes.

\subsection{Preliminaries: Approximation of discrete characteristic functions}
Ideally, to obtain a stability estimate at arbitrary $t \in (t^{n-1},t^n]$, we would like to substitute $u_h =
\chi_{[t^{n-1},t)} u_h$ into the discrete
equations (\ref{eqn:dac}). However, this choice is not available in the discrete setting, since $
\chi_{[t^{n-1},t)} u_h$ is not a member of ${\mathcal U}_h$,
unless $t$ coincides with a partition point. Therefore,
approximations of such functions need to be constructed. This is
done in \cite[Section 2.3]{ChWa06}. For completeness we state the
main results. The approximations are constructed on the interval
$(0,\tau)$, and they are invariant under translations.
For fixed (but arbitrary) $t \in (0,\tau)$ let $p \in {\mathcal P}_k (0,\tau)$, and
denote the discrete approximation of $\chi_{[0,t)} p$ by the
polynomial $ \tilde p \in {\mathcal P}_k(0,\tau)$ with,
$\tilde p(0)=p(0)$ which satisfies
$$ \int_0^\tau \tilde p q = \int_0^t pq \qquad\forall\, q \in
{\mathcal P}_{k-1} (0,\tau). $$

To motivate the above
construction we simply observe that for
$q=p^\prime$ we obtain $\int_0^\tau p^\prime \tilde p = \int_0^t p
p^\prime = \frac{1}{2} (p^2(t)-p^2(0))$.

It is clear that this construction can be extended to
approximations of $\chi_{[0,t)} u$ for
$u \in {\mathcal P}_k[0,\tau;U]$ where $U$ is
a linear space. Note that if $u \in {\mathcal P}_{k}[0,\tau;U]$ then it can be written as $u = \sum_{i=0}^k p_i(t)u_i$ where $p_i \in {\mathcal P}_k[0,\tau]$ and ${u_i} \in U$. The discrete approximation of $\chi_{[0,t)} u$ in
${\mathcal P}_k[0,\tau;U]$ is then defined by $\tilde u = \sum_{i=0}^k
\tilde p_i(t) u_i$ and if $U$ is a semi-inner product space we deduce,
$$ \tilde u(0) = u(0), \quad\mbox{ and} \int_0^\tau (\tilde u,w)_U
= \int_0^t (u,w)_{U} \quad\forall w \in {\mathcal
P}_{k-1}[0,\tau;U].$$
It remains to quote the main results from \cite{ChWa06,ChWa10,Wa10}.

\begin{prop} \label{prop:char}
Suppose that $U$ is a (semi) inner product space. Then the mapping
$\sum_{i=0}^k p_i(t)u_i$ $\rightarrow$ $\sum_{i=0}^k \tilde p_i(t)
u_i$ on ${\mathcal P}_k[0,\tau;U]$ is continuous in
$\|.\|_{L^2[0,\tau;U]}$. In particular,
$$ \|\tilde u\|_{L^2[0,\tau;U]} \leq C_k \|u\|_{L^2[0,\tau;U]}, \qquad
\|\tilde u - \chi_{[0,t)}u\|_{L^2[0,\tau;U]} \leq C_k
\|u\|_{L^2[0,\tau;U]}$$ where $C_k$ is a constant depending on
$k$.
\end{prop}
\begin{proof} See \cite[Lemma 2.4]{ChWa06}.
\end{proof}

A standard calculation gives an explicit formula of
$\tilde u = \rho(s) z$, when we choose $u(s) = z \in U$ to be constant (see
e.g. \cite{ChWa10}).
\begin{lem} \label{lem:char1}
Fix $t \in [0,\tau]$ and let $\rho \in {\mathcal P}_k[0,\tau]$
characterized by
$$ \rho(0)=1, \qquad \int_0^\tau \rho q = \int_0^t q , \qquad\, q \in
{\mathcal P}_{k-1}[0,\tau].$$ Then,
$$ \rho(s) = 1 + (s/\tau) \sum_{i=0}^{k-1} c_i \hat p_i (s/\tau),
\qquad c_i = \int_{t/\tau}^1 \hat p_i(\eta) d\eta, $$ where
$\{\hat p_i\}_{i=0}^{k-1}$ is an orthonormal basis of ${\mathcal
P}_{k-1}[0,1]$ in the (weighted) space $L^2_w[0,1]$ having inner
product
$$ (\hat p, \hat q) = \int_0^1 \eta \hat p (\eta) \hat q(\eta)
d\eta.$$ In particular, $\|\rho\|_{L^{\infty}(0,\tau)} \leq C_k,$
where $C_k$ is independent of $t \in [0,\tau]$.
\end{lem}

\subsection{The main stability estimate at arbitrary time points}

Now, we are ready to state the main stability result at arbitrary time-points which plays a key role to the derivation of best approximation estimates. We emphasize that the time-discretization parameter
$\tau$ is chosen independent of $h$ and the dependence
of the stability constant upon $1/\epsilon$ is polynomial.
\begin{prop} \label{prop:stabinfty}
Suppose that $f \in L^2[0,T;H^{-1}(\Omega)]$, $u_0 \in
L^2(\Omega)$, and let $u_h$ be the approximate solution
computed by using the discontinuous time-stepping scheme. Then,
there exists constant $C$ depending on $\Omega$, $C_k$ and $T$ (but not $\epsilon$), such that
$$ \|u_h\|_{L^{\infty}[0,T;L^2(\Omega)]} \leq C (1/\epsilon).$$
\end{prop}
\begin{proof} Recall that setting $w_h = u_h$, in (\ref{eqn:dac}), using Poincar{\'e} and Young's inequality, we obtain
respectively
\begin{eqnarray} \label{eqn:stab2}
&&(1/2) \|u^n_{h-}\|^2_{L^2(\Omega)} - (1/2)
\|u^{n-1}_{h-}\|^2_{L^2(\Omega)} + (1/2) \| [u^{n-1}_h]\|^2_{L^2(\Omega)}\nonumber  \\
&&+ \int_{t^{n-1}}^{t^n} \left (
  (C/2) \|u_h\|^2_{H^1(\Omega)} + (1/\epsilon^2) \|u_h\|^{4}_{L^{4}(\Omega)}
  \right ) dt  \nonumber \\
&& \leq
(1/\epsilon^2) \int_{t^{n-1}}^{t^n}
 \|u_h\|^2_{L^2(\Omega)} dt
       + \int_{t^{n-1}}^{t^n}  (2/C) \|f\|^2_{H^{-1}(\Omega)}
       dt. \nonumber
\end{eqnarray}
In order to avoid the use of a Gr\"onwall type argument, we will need to estimate the term
$(1/\epsilon^2) \int_{t^{n-1}}^{t^n} \|u_h\|^2_{L^2(\Omega)}$ using the approximation of the discrete characteristic.
We employ properties of the
discrete characteristic and its approximation by following the
technique of \cite{ChWa10} and the stability estimates of Lemma \ref{lem:stab}. For fixed $t \in [t^{n-1},t^n)$
and $z_h \in U_h$ we substitute $w_h(s) = z_h \rho (s)$ into
(\ref{eqn:dac}), where $\rho (s) \in {\mathcal P}_k[t^{n-1},t^n]$
is constructed similar to Lemma \ref{lem:char1}, i.e.,
$$ \rho (t^{n-1}) = 1, \qquad \int_{t^{n-1}}^{t^n} \rho q =
\int_{t^{n-1}}^t q, \qquad\, q \in {\mathcal
P}_{k-1}[t^{n-1},t^n].$$ Recall that Lemma \ref{lem:char1} asserts
that $\|\rho\|_{L^{\infty}(t^{n-1},t^n)} \leq C_k$, with $C_k$ independent of
$t$. Now, it is easy to see that with this particular choice of
$w_h$,
\begin{eqnarray*}
&& \int_{t^{n-1}}^{t^n} (u_{ht},w_h)ds +
(u^{n-1}_{h+}-u^{n-1}_{h-},w^{n-1}_{h+}) \\
&&=\int_{t^{n-1}}^{t}
(u_{ht},z_h)ds + (u^{n-1}_{h+}-u^{n-1}_{h-},\rho(t^{n-1}) z_h )  = (u_{h}(t)-u^{n-1}_{h-},z_h).
\end{eqnarray*}
Hence integrating by parts (in time) equation (\ref{eqn:dac}) and
using the above computation, we obtain
\begin{eqnarray*}
&& (u_{h}(t)-u^{n-1}_{h-},z_h) \\
&& = - \int_{t^{n-1}}^{t^n} \left (
a(u_h,z_h \rho) + (1/\epsilon^2) (u^3_h-u_h,z_h \rho) \right ) ds
+ \int_{t^{n-1}}^{t^n} \langle f,z_h \rho \rangle ds \\
&& \leq  C_k \Big [ \int_{t^{n-1}}^{t^n} \|\nabla
u_h\|_{L^2(\Omega)} \|\nabla z_h\|_{L^2(\Omega)} + \int_{t^{n-1}}^{t^n} \|f\|_{H^{-1}(\Omega)}
\|z_h\|_{H^1(\Omega)} ds \\
&& + (1/\epsilon^2) \int_{t^{n-1}}^{t^n} \left (
\|u^3_h\|_{L^{4/3}(\Omega)}\|z_h\|_{L^4(\Omega)} +
\|u_h\|_{L^2(\Omega)} \|z_h\|_{L^2(\Omega)}  \right ) ds  \Big ] ,
\end{eqnarray*}
where we have used Lemma \ref{lem:char1} to bound
$\|\rho\|_{L^{\infty}(t^{n-1},t^n)} \leq C_k$ with $C_k$ denoting a constant
depending only on $k$, $\Omega$. Note also that $z_h \in U_h$ (independent of $s$),
hence the above inequality leads to
\begin{eqnarray*}
&&(u_{h}(t)-u^{n-1}_{h-},z_h) \leq  C_k \left [ \int_{t^{n-1}}^{t^n}
\big (\|u_h\|_{H^1(\Omega)}  +
\|f\|_{H^{-1}(\Omega)} \big ) ds \right ] \|z_h\|_{H^1(\Omega)}
\\
&&+C_k (1/\epsilon^2) \Big ( \left [  \int_{t^{n-1}}^{t^n}
\|u_h\|^3_{L^{4}(\Omega)} ds \right ] \|z_h\|_{L^4(\Omega)} + \left [ \int_{t^{n-1}}^{t^n} \|u_h\|_{L^2(\Omega)} ds \right ] \|z_h\|_{L^2(\Omega)} \Big ).
\end{eqnarray*}
Here we have used the fact $\|u^3_h\|_{L^{4/3}(\Omega)} =
\|u_h\|^3_{L^4(\Omega)}$.
Setting $z_h = u_h(t)$ (for the previously fixed $t \in [t^{n-1},t^n)$), using
H\"older's inequality, and integrating in time the resulting
inequality, we obtain,
\begin{eqnarray} \label{eqn:stab3}
&& \int_{t^{n-1}}^{t^n}
\|u_h(t)\|^2_{L^2(\Omega)} dt
\leq  \|u^{n-1}_{h-}\|_{L^2(\Omega)} \tau^{1/2}_n \|u_h(t)\|_{L^2[t^{n-1},t^n;L^2(\Omega)]}    \nonumber \\
&&+ C_k \tau^{1/2}_n \Big (
\|u_h\|_{L^2[t^{n-1},t^n;H^1(\Omega)]} +
\|f\|_{L^2[t^{n-1},t^n;H^{-1}(\Omega)]} \Big )
\int_{t^{n-1}}^{t^n} \|u_h(t)\|_{H^1(\Omega)}dt \nonumber \\
&&+C_k \tau^{1/4}_n (1/\epsilon^2) \left (
\|u_h\|^3_{L^4[t^{n-1},t^n;L^4(\Omega)]} \right )\int_{t^{n-1}}^{t^n}
\|u_h(t)\|_{L^4(\Omega)}dt \nonumber \\
&&+ C_k \tau^{1/2}_n (1/\epsilon^2) \|u_h\|_{L^2[t^{n-1},t^n;L^2(\Omega)]} \int_{t^{n-1}}^{t^n}
\|u_h(t)\|_{L^2(\Omega)}dt.
\end{eqnarray}
H\"older's inequality implies that
$\int_{t^{n-1}}^{t^n} \|u_h\|_{L^4(\Omega)} dt \leq \tau^{3/4}_n \|u_h\|_{L^4[t^{n-1},t^n;L^4(\Omega)]}$, and $\int_{t^{n-1}}^{t^n} \|u_h\|_{H^1(\Omega)} dt  \leq \tau^{1/2}_n \|u_h\|_{L^2[t^{n-1},t^n;H^1(\Omega)]}$. Therefore,
using Young's inequalities we deduce (with different $C_k$),
\begin{eqnarray}
\label{eqn:stab3b} && (1/2) \int_{t^{n-1}}^{t^n}
\|u_h(t)\|^2_{L^2(\Omega)} dt
\leq  (\tau_n /2) \|u^{n-1}_{h-}\|^2_{L^2(\Omega)} \nonumber \\
&&+ C_k \tau_n \left (
\|u_h\|^2_{L^2[t^{n-1},t^n;H^1(\Omega)]} +
\|f\|^2_{L^2[t^{n-1},t^n;H^{-1}(\Omega)]} \right ) \nonumber \\
&& +C_k \tau_n (1/\epsilon^2) \left ( \|u_h\|^4_{L^4[t^{n-1},t^n;L^4(\Omega)]} + \|u_h\|^2_{L^2[t^{n-1},t^n;L^2(\Omega)]} \right ).
\end{eqnarray}
Now, using an inverse estimate, $\|u_h(t)\|^2_{L^2(\Omega)} \leq (C_k / \tau_n) \int_{t^{n-1}}^{t^n} \|u_h\|^2_{L^2(\Omega)}$, we obtain,
 \begin{eqnarray}
 \label{eqn:stab5b} &&
\|u_h(t)\|^2_{L^2(\Omega)}
\leq  C_k \Big [ \|u^{n-1}_{h-}\|^2_{L^2(\Omega)} +
\|u_h\|^2_{L^2[t^{n-1},t^n;H^1(\Omega)]} +
\|f\|^2_{L^2[t^{n-1},t^n;H^{-1}(\Omega)]} \nonumber \\
&& +(1/\epsilon^2) \left ( \|u_h\|^4_{L^4[t^{n-1},t^n;L^4(\Omega)]} + \|u_h\|^2_{L^2[t^{n-1},t^n;L^2(\Omega)]} \right ) \Big ] . \nonumber
\end{eqnarray}
The proof now follows by simply substituting the previously developed bounds of (\ref{lem:stab}).
\end{proof}

\begin{rmk}
The above theorem states that the discontinuous Galerkin
discretization inherits the stability estimates of the weak
formulation under minimal regularity assumptions on the given
data. This is an important asset related to the discontinuous (in time) Galerkin formulation.
\end{rmk}

\subsection{Convergence under minimal regularity assumptions}
We quote a discrete compactness argument of Walkington (see \cite[Theorem 3.1]{Wa10}) which allows to recover strong convergence in an appropriate norm, and pass the limit through the
nonlinear term. The compactness argument combined with the stability estimates of Lemma \ref{lem:stab} and Proposition \ref{prop:stabinfty}, imply the convergence of the space-time
approximations under minimal regularity assumptions, while the dependence upon $(1/\epsilon)$ does not deteriorate any further.

The compactness argument concerns numerical approximations of solutions
$u:[0,T] \to U$ of general evolution equations of the form
\begin{equation} \label{eqn:ap1}
u_t + A(u) = f(u), \qquad u(0)=u_0,
\end{equation}
where $U$ is a Banach space and each term of the equation takes
values in $U^*$. Both $A(u)=A(t,u)$ and $f(u)=f(t,u)$ may
depend upon $t$ and are allowed to be nonlinear, however, in our
setting only $f(u) \equiv - (1/\epsilon^2) (u^3-u)$ contains nonlinear terms. Suppose that $U \subset H \subset U^*$ (with continuous embeddings) form the
standard evolution triple, i.e., the pivot space $H$ is a Hilbert
space. The numerical schemes approximate the weak form of
(\ref{eqn:ap1}), i.e.,
\begin{equation} \label{eqn:ap2}
\langle u_t,w \rangle + a(u,w) = \langle f(u),w \rangle , \qquad\forall\, w \in U \,
\end{equation}
where $a: U \times U \to \mathbb R$ is defined by $a(u,w)=
(A(u),w)$. 
Set $F(u) \equiv
f(u)-A(u)$. Then the following theorem \cite[Theorem 3.1]{Wa10}
establishes the compactness property of the discrete approximation.

\begin{thm} \label{thm:Walk}
Let $H$ be a Hilbert space, $U$ be a Banach space and $U \subset H
\subset U^*$ be dense and compact embeddings. Fix an integer $k\geq
0$ and let $1\leq p,q <\infty$. Let $h>0$ be the mesh parameter, and
let $\{t^i\}_{i=0}^N$ denote a quasi-uniform partition of $[0,T]$. Let $U_h \subset U$ denote standard finite element spaces.
Assume that
\begin{enumerate}
\item For each $h,\tau >0$, $u_h \in \{u_h \in L^p[0,T;U] \quad | \quad
u_h|_{(t^{n-1},t^n)} \in {\mathcal P}_k[t^{n-1},t^n;U_h] \}$ and
on each interval, satisfies
$$ \int_{t^{n-1}}^{t^n} \langle u_{ht},w_h \rangle dt
+(u^{n-1}_{h+}-u^{n-1}_{h-},w^{n-1}_{h+}) =
\int_{t^{n-1}}^{t^n} \langle F(u_h),w_h \rangle dt $$ for every $w_h \in
{\mathcal P}_k[t^{n-1},t^n;U_h]$. \item $\{u_h\}_{h>0}$ is bounded
in $L^p[0,T;U]$ and $\{ \|F(u_h)\|_{L^q[0,T;U^*]} \}_{h>0}$ is
also bounded.
\end{enumerate}
Then,
\begin{enumerate}
\item If $p>1$ then $\{u_h\}_{h>0}$ is compact in $L^r[0,T;H]$ for
$1\leq r <2p$. \item If $1 \leq (1/p)+(1/q) <2$, and
$\sum_{i=1}^{N} \|[u^i_h]\|^2_{H} < C$ is bounded independent of
$h$, then $\{u_h\}_{h>0}$ is compact in $L^r[0,T;H]$ for $1 \leq r
< 2/((1/p)+(1/q)-1)$.
\end{enumerate}
\end{thm}
\begin{proof} See \cite[Theorem 3.1]{Wa10}.
\end{proof}

We will utilize the above result to obtain strong convergence of the
discrete Allen-Cahn equation to the continuous one. The lack of any  meaningful regularity for the discrete time derivative due to the presence of discontinuities, requires special attention since the  classical Aubin-Lions compactness argument is not directly applicable.
\begin{thm} \label{thm:conv}
Suppose that $f \in L^2[0,T;H^{-1}(\Omega)]$, $u_0 \in L^2(\Omega)$, and
let $\epsilon <1$ be a given parameter. Let
$\{t^i\}_{i=0}^N$ denote a quasi-uniform partition of $[0,T]$.
Suppose that the assumptions of Proposition \ref{prop:stabinfty} hold, and let $\tau, h \to
0$. Then, the following convergence results hold:
$$
u_{h} \to u \mbox{ weakly in $L^2[0,T;H^1_0(\Omega)]$}, \quad
u_{h} \to u \mbox{ weakly-* in
$L^{\infty}[0,T;L^2(\Omega)]$},
$$
and
$$
u_h \to u \quad\mbox{ strongly in $L^r[0,T;L^2(\Omega)]$, for every $ 1 \leq r < \infty$}.$$
In addition $u$ is a weak
solution of the Allen-Cahn equation.
\end{thm}
\begin{proof}
We follow the same arguments with \cite[Section 6]{Wa10}.
The stability estimates of Lemma \ref{lem:stab} and Proposition \ref{prop:stabinfty}, imply
(passing to a subsequence if necessary) there exists $u$ such that
$u_h \to u$ weakly in $L^2[0,T;H^1_0(\Omega)]$ and
weakly-* in $L^{\infty}[0,T;L^2(\Omega)]$. We note that $\{u_h\}_{h,\tau}$ is bounded independent of $\tau,h,\epsilon$ in $L^2[0,T;L^2(\Omega)]$ and $L^4[0,T;L^4(\Omega)]$. It remains to obtain strong convergence
in $L^2[0,T;L^2(\Omega)]$. For this purpose, fix $U=H^1_0(\Omega)$, $H=L^2(\Omega)$,
and $F(u)=  \Delta u- (1/\epsilon^2) \left ( u^3-u \right ) - f$. It is easy to show that
$F(u_h) \in L^{4/3}[0,T;H^{-1}(\Omega)]$. Indeed,
$u_h \in L^2[0,T;H^1_0(\Omega)] \cap L^4[0,T;L^4(\Omega)]$,
and $u_h \in L^{\infty}[0,T;L^2(\Omega)]$
clearly imply that $u^3_h \in L^{4/3}[0,T;H^{-1}(\Omega)]$ by using
standard interpolation theorems. The remaining terms can be handled easily. Note also that $\sum_{i=1}^N \|[u^i_h]\|^2_{L^2(\Omega)} \leq C$ where $C$ is independent of $\tau,h$.
Therefore, using the Theorem \ref{thm:Walk}, we obtain the desired strong convergence in $L^2[0,T;L^2(\Omega)]$.
Choose $w_h \in C[0,T;U_h] \cap
{\mathcal U}_h$, with $w_h(T)=0$. Then, summing equations (\ref{eqn:dac}) from
$n=1$ to $n=N$, we deduce that
\begin{eqnarray*}
&& (u_h(T),w_h(T)) + \int_{0}^{T}
        \left ( - \langle u_h,w_{ht} \rangle +  a(u_h,w_h) +
        (1/\epsilon^2) \langle u^3_h-u_h,w_h \rangle \right )dt
        \label{eqn:dgs1} \\
&& \qquad\qquad  = \int_{0}^{T}  \langle f,w_h \rangle dt + (u^0,w_h(0)).
\end{eqnarray*}
Note that we may pass the limit through the linear terms due to the
stability estimates on $u_h$ and the fact that
$w_h \in C[0,T;U_h] \cap {\mathcal U}_h$. The semi-linear term can be treated
by the strong convergence on $L^2[0,T;L^2(\Omega)]$. Indeed, using Holder's inequality, Landyzeskaya-Gagliardo-Nirenberg interpolation inequality,
\begin{eqnarray*}
&& \int_0^T \left | \langle u^3_h-u^3,w_h \rangle \right | dt \leq
\int_0^T \left | \langle (u_h-u)(u^2_h+u^2+u_h u),w_h \rangle \right | dt \\
&& \leq  C \int_0^T \|u_h-u\|_{L^3(\Omega)} (\|u_h\|^2_{L^4(\Omega)} + \|u\|^2_{L^4(\Omega)}) \|w_h\|_{L^6(\Omega)} dt \\
&& \leq C \|w_h\|_{C[0,T;H^1(\Omega)]} \int_0^T \|u_h-u\|^{1/2}_{L^2(\Omega)} \|u_h-u\|^{1/2}_{H^1(\Omega)}
(\|u_h\|^2_{L^4(\Omega)} + \|u\|^2_{L^4(\Omega)})dt \\
&& \leq C \|w_h\|_{C[0,T;H^1(\Omega)]} \|u_h-u\|^{1/2}_{L^2[0,T;L^2(\Omega)]} \|u_h-u\|^{1/2}_{L^2[0,T;H^1(\Omega)]} \\
&& \qquad \times (\|u_h\|_{L^4[0,T;L^4(\Omega)]} + \|u\|_{L^4[0,T;L^4(\Omega)]} )^2.
\end{eqnarray*}
A standard density argument, now completes the proof.
\end{proof}
The unconditional stability estimates and the above convergence result, validate the use of discontinuous Galerkin time-stepping schemes of order $k \geq 1$. In particular, for any $\alpha >0$ there exist $\tilde h, \tilde \tau$ such that, for every $\tau \leq \tilde \tau$ and $h \leq \tilde h$, we obtain, $\|u_h-u\|_{L^4[0,T;L^2(\Omega)]} \leq \alpha.$ For the error estimates, we will choose to work with $\tau,h$  (chosen independently) such that, 
\begin{equation} \label{eqn:conv}
\left\{ \begin{array}{rl} 
& \hskip-15pt \|u_h-u\|_{L^4[0,T;L^2(\Omega)]} \leq \delta \epsilon^4, \mbox{ when $d=3$, for  $(\tau,h)$ satisfying $\tau \leq \tilde \tau$, $h \leq \tilde h$}, \\
& \hskip-15pt \|u_h-u\|_{L^4[0,T;L^2(\Omega)]} \leq \delta \epsilon^{7/2}, \mbox{ when $d=2$, for $(\tau,h)$ satisfying $\tau \leq \tilde \tau$, $h \leq \tilde h$}, \\
& \hskip-15pt \|u_h-u\|_{L^4[0,T;L^2(\Omega)]} \leq \delta \epsilon^3, \mbox{ when $d=2$, $k=0,1$, for $(\tau,h)$ satisfying $\tau \leq \tilde \tau$, $h \leq \tilde h$},
\end{array} \right.
\end{equation}
where $\delta >0$ (to be chosen later) is independent of $\epsilon$. Note that due to the unconditional stability in $L^4[0,T;L^4(\Omega)]$ with bounds independent of $\epsilon$, $\tilde \tau, \tilde h$ can be chosen independent of $\epsilon$. 
We close this Section by a short remark regarding the computation of such discrete solution.
\begin{rmk}
 It is expected that at least for moderate values of the papemeter $\epsilon$, even when $\tau \approx h$, the computation of the fully-discrete solution follow by using techniques established for the numerical solution of linear and semi-linear parabolic pdes by discontinuous time-stepping schemes. However, when using high order schemes, due to the large and non-symmetric structure of the associated system, special attention is necessary. For specialized preconditioners for high-order discontinuous Galerkin schemes, we refer the reader to the recent work of \cite{Sm16}, where various issues regarding robustness and efficiency of preconditioners suitably constructed for discontinuous Galerkin time-stepping methods are being discussed. 
\end{rmk}
\section{Error estimates}
\setcounter{equation}{0}

\subsection{Preliminary Estimates}
The following projections related to discontinuous Galerkin time-stepping schemes will be used.
\begin{defn} \label{defn:4.1}
(1) The projection $\bbP_n^{loc}: C[t^{n-1},t^n ;L^2(\Omega)]
\rightarrow {\mathcal P}_k[t^{n-1},t^n; U_h]$ satisfies $
(\bbP_n^{loc} w)^n = P_h w(t^n), $ and
\begin{equation*} \label{eqn:4.3}
\int_{t^{n-1}}^{t^{n}} (w- \bbP_n^{loc}w,W_h) = 0, \qquad \forall
\, W_h \in {\mathcal P}_{k-1} [t^{n-1},t^n;U_h].
\end{equation*}
In the above definition, we have used the convention $(\bbP_n^{loc} w)^n \equiv (\bbP_n^{loc} w)(t^n)$, and $P_h: L^2(\Omega) \rightarrow U_h$ is
the orthogonal $L^2$ projection operator onto $U_h \subset
H^1(\Omega)$.

(2) The projection $\bbP^{loc}_h: C[0,T;L^2(\Omega)] \rightarrow
{\mathcal U}_h$ satisfies
$$
\bbP^{loc}_h w \in {\mathcal U}_h \mbox{ and } (\bbP^{loc}_h
w)|_{(t^{n-1},t^n]} = \bbP_n^{loc} (w|_{[t^{n-1},t^n]}).
$$
\end{defn}

In the following Lemma, we collect several results regarding (optimal) rates of convergence for the above projection (see e.g. \cite{ChWa10}).  

\begin{lem} \label{lem:rates1}
Let $U_h \subset H^1(\Omega)$, and $\bbP^{loc}_h$ defined in Definition \ref{defn:4.1} respectively. Then, for all $w \in L^2[0,T;H^{l+1}(\Omega)] \cap H^{k+1}[0,T;L^2(\Omega)]$
there exists constant $C \ge 0$ independent of $h,\tau$ such that
$$
\|w-\bbP^{loc}_h w\|_{L^2[0,T;L^2(\Omega)]} \leq C \big (
h^{l+1} \|w\|_{L^2[0,T;H^{l+1}(\Omega)]} + \tau^{k+1} \|w^{(k+1)}\|_{L^2[0,T;L^2(\Omega)]} \big ),
$$
$$ \|w-\bbP^{loc}_hw\|_{L^2[0,T;H^1(\Omega)]} \leq C \big (
h^l \|w\|_{L^2[0,T;H^{l+1}(\Omega)]} + (\tau^{k+1} /h) \|w^{(k+1)}\|_{L^2[0,T;L^2(\Omega)]} \big ),
$$
$$
\|w-\bbP^{loc}_h w\|_{L^{\infty}[0,T;L^2(\Omega)]} \leq C \big (
h^{l+1} \|w\|_{L^{\infty}[0,T;H^{l+1}(\Omega)]} + \tau^{k+1} \|w^{(k+1)}\|_{L^{\infty}[0,T;L^2(\Omega)]} \big ).
$$
Let $k=0,l=1$, and $w \in L^2[0,T;H^2(\Omega)] \cap H^1[0,T;L^2(\Omega)]$. Then there exists constant $C \ge 0$ independent of $h,\tau$ such that,
\begin{eqnarray*}
&& \|w-\bbP^{loc}_hw\|_{L^{\infty}[0,T;L^2(\Omega)]} + \|w-\bbP^{loc}_hw\|_{L^2[0,T;H^1(\Omega)]} \leq C \big (
h\|w\|_{L^2[0,T;H^{2}(\Omega)]} \\
&& \qquad  + \tau^{1/2} (\|w_t\|_{L^2[0,T;L^2(\Omega)]} + \|w\|_{L^2[0,T;H^2(\Omega)]} ) \big ).
\end{eqnarray*}
\end{lem}
\begin{rmk} \label{rmk:regproj}
If more regularity (in time) is available then the above estimates can be improved. In particular, if  $w \in L^2[0,T;H^{l+1}(\Omega)] \cap H^{k+1}[0,T;H^1(\Omega)]$, then we obtain, 
$$ \|w-\bbP^{loc}_hw\|_{L^2[0,T;H^1(\Omega)]} \leq C \big (
h^l \|w\|_{L^2[0,T;H^{l+1}(\Omega)]} + \tau^{k+1} \|w^{(k+1)}\|_{L^2[0,T;H^1(\Omega)]} \big ).
$$
\end{rmk}

The fully-discrete Galerkin orthogonality can be written as follows: Subtracting (\ref{eqn:dac}) from (\ref{eqn:wac2}), we obtain for every $w_h \in {\mathcal U}_h$ and for $n=1,...,N$,
\begin{eqnarray}
&& (e^n_{-},w^n_{h-}) + \int_{t^{n-1}}^{t^n}
        \left (  - \langle e,w_{ht} \rangle +  a(e,w_h) \right ) dt \label{eqn:orth} \\
&& \qquad  +
        (1/\epsilon^2)  \int_{t^{n-1}}^{t^n} \left ( (u^3_h-u^3,w_h)-(u_h-u,w_h) \right  ) dt
        =  (e^{n-1}_{-},w^{n-1}_{h+}) \nonumber
\end{eqnarray}
where $e =u_h - u$ denotes the error. We will split the error as $e= (u_h - u_p) + (u_p -u) \equiv e_h + e_p$, where $u_p$ is the discontinuous Galerkin solution of a linear parabolic pde with right hand side $u_t - \Delta u$, and initial data $u_{p0} = P_hu_0$,  i.e.,  for every $w_h \in {\mathcal U}_h$ and for $n=1,...,N$, $u_p \in {\mathcal U}_h$ is the solution of,
\begin{eqnarray}
(u^n_{p-},w^n_{h-}) &+& \int_{t^{n-1}}^{t^n}
        \Big(- \langle u_p,w_{ht} \rangle + a(u_p,w_h) \Big)dt
        \label{eqn:proju} \\ \nonumber
        &=& (u^{n-1}_{p+},w^{n-1}_{+}) + \int_{t^{n-1}}^{t^n} \langle u_t- \Delta u, w_h \rangle dt.
\end{eqnarray}
Integrating by parts the last term of the right hand side, we obtain the orthogality condition: For $n=1,...,N$, and $w_h \in {\mathcal U}_h$
\begin{equation} \label{eqn:eporth}
(e^n_{p-},w^n_{h-}) + \int_{t^{n-1}}^{t^n}
        \Big(- \langle e_p,w_{ht} \rangle + a(e_p,w_h) \Big)dt
        = (e^{n-1}_{p+},w^{n-1}_{+}).
\end{equation}
The following best approximation estimates under minimal regularity assumptions that bound the error $e_p=u_p-u$ in terms of the local projections of Definition \ref{defn:4.1} are straightforward application of \cite[Theorem 2.2 and Theorem 2.3]{ChWa06}).
\begin{eqnarray} \label{eqn:projest1}
&& \|e_p\|_{L^{\infty}[0,T;L^2(\Omega)]} + \|e_p\|_{L^2[0,T;H^1(\Omega)]}  \leq C \Big ( \|P_hu(0)-u(0)\|_{L^2(\Omega)} \\
&& \qquad + \|u-\bbP^{loc}_h u\|_{L^{\infty}[0,T;L^2(\Omega)]} + \|u-\bbP^{loc}_h u\|_{L^2[0,T;H^1(\Omega)]} \Big ), \nonumber 
\end{eqnarray}
where $C$ is a constant depending upon $\Omega$ and the constant $C_k$ of Proposition \ref{prop:char}. 
In addition, 
\begin{equation} \label{eqn:upstab}
\|u_p\|_{L^{\infty}[0,T;H^1(\Omega)]} \leq C(\|u_0\|_{H^1(\Omega)} + \|u_t-\Delta u\|_{L^2[0,T;L^2(\Omega)]}).
\end{equation}
by \cite[Theorem 4.10]{ChWa10}. 
Returning back to the orthogonality condition (\ref{eqn:orth}) and using (\ref{eqn:eporth}) we obtain, the following relation for $e_h =u_h - u_p$: For all $w_h \in {\mathcal U}_h$ and for $n=1,...,N$, 
\begin{eqnarray}
&& (e^n_{h-},w^n_{h-}) + \int_{t^{n-1}}^{t^n}
        \left (  - \langle e_h,w_{ht} \rangle +  a(e_h,w_h) \right ) dt \\
&& \qquad  +
        (1/\epsilon^2)  \int_{t^{n-1}}^{t^n} \left ( (u^3_h-u^3,w_h)-(u_h-u,w_h) \right  ) dt
        \label{eqn:orth2} =  (e^{n-1}_{h-},w^{n-1}_{h+}). \nonumber
\end{eqnarray}
Adding and subtracting the term $u^3_p$ in the nonlinear term, we equivalently obtain,
\begin{eqnarray}
&& (e^n_{h-},w^n_{h-}) + \int_{t^{n-1}}^{t^n}
        \left (  - \langle e_h,w_{ht} \rangle +  a(e_h,w_h) \right ) dt  -  (e^{n-1}_{h-},w^{n-1}_{h+}) \nonumber \\
&&  +   (1/\epsilon^2)  \int_{t^{n-1}}^{t^n} \left ( (u^3_h-u^3_p,w_h)-(e_h,w_h) \right  ) dt  \nonumber \\
 && =      (1/\epsilon^2)  \int_{t^{n-1}}^{t^n} \left ( u^3_p-u^3,w_h)-(e_p,w_h) \right  ) dt.    \label{eqn:orth3} 
\end{eqnarray}
Our focus is to bound $e_h$ in terms of $e_p$ without introducing constants that depend exponentially upon $1/\epsilon$. 

To simplify the presentation, we will denote by
$C_{\infty} = \|u\|_{L^{\infty}[0,T;L^{\infty}(\Omega)]}$, and we note that if in addition to (\ref{eqn:regass}), $u_0 \in L^{\infty}(\Omega)$, with norm bounded independent of $\epsilon$ then $C_{\infty}$ is also bounded independent of $\epsilon$. We first recall the spectral estimate of \cite{MoSc95}, which states that
if $u$ is solution of (\ref{eqn:ac}) then there exists a positive constant $C_s$ independent of $\epsilon$ such that,
\begin{equation} \label{eqn:spec}
\inf_{ \phi \in H^1(\Omega), \phi \neq 0} \frac{ \|\nabla \phi\|^2_{L^2(\Omega)} + (1/ \epsilon^2)  \left ( (3u^2-1)\phi,\phi \right )} { \|\phi\|^2_{L^2(\Omega)}} \geq - C_s.
\end{equation}  

We follow the approach presented in Section 3. In particular, given right hand side $e_h \in L^{\infty}[0,T;L^2(\Omega)]$, and terminal data $\psi^N_{h+}=0$, we seek $\psi_h \in {\mathcal U}_h$ such that for all $w_h \in {\mathcal P}_k[t^{n-1},t^n;U_h]$, and for all $n=N,...,1$, 
\begin{eqnarray} \label{eqn:orth4}
&& -({\psi}^n_{h+},w^n_{h-}) + \int_{t^{n-1}}^{t^n} \big ( (\psi_h,w_{ht})
+ a({\psi}_h,w_h)  \big )dt  +({\psi}^{n-1}_{h+},w^{n-1}_{h+}) \nonumber \\
&& + \frac{1}{\epsilon^2} \int_{t^{n-1}}^{t^n} \left ( 3u^2 \psi_h,w_h \right ) - \frac{1}{\epsilon^2} \int_{t^{n-1}}^{t^n} ( \psi_h,w_h ) dt  = \int_{t^{n-1}}^{t^n} (e_{h},w_h)dt.
\end{eqnarray}
Note that despite the fact that the above pde is linearized analog of the Allen-Cahn equation, and the spectral estimate can be applied directly to obtain a preliminary bound on the energy norm and at arbitrary time-points when $k=0,1$.  


\begin{lem} \label{lem:auxpsi}
Let $e_h \in L^2[0,T;L^2(\Omega)]$, and $u \in L^4[0,T;L^4(\Omega)]$ with bounds independent of $\epsilon$. Then, for $\tau_n \leq C_k \frac{\epsilon^3}{\|u\|_{L^{\infty}[0,T;L^6(\Omega)]}} $, $\psi_h \in {\mathcal U}_h$ satisfies for all $n=N,...,1$,
\begin{eqnarray*}
&& \|\psi^{n-1}_{h+}\|_{L^2(\Omega)}  + \|\psi_h\|_{L^2[0,T;L^2(\Omega)]}  + \|u\psi_h\|_{L^2[0,T;L^2(\Omega)]} + \epsilon \| \psi_h\|_{L^2[0,T;H^1(\Omega)]} \\
&&  \leq C \|e_h\|_{L^2[0,T;L^2(\Omega)]}, \\
&&  \|\psi_h\|_{L^{\infty}[0,T;L^2(\Omega)]} \leq \frac{C}{\epsilon} \|e_h\|_{L^2[0,T;L^2(\Omega)]}. 
\end{eqnarray*}
where the constants $C$ depend only upon $C_s$, the domain, the constant $C_k$ of Lemma 4.2 and the data $f,u_0$ (through the norms of $\|u\|_{L^4[0,T;L^4(\Omega)]}$), and are independent of $\tau,h,\epsilon$. In addition, there exists a costant $C$ depending upon $C_s$, the domain, the constant $C_k$ of Lemma 4.2,  and the norm 
$\|u\|_{L^{\infty}[0,T;L^{\infty}(\Omega)]}$ such that, 
\begin{eqnarray*}
&& \|\psi_h\|_{L^{\infty}[0,T;H^1(\Omega)]} + \|\Delta_h \psi_h\|_{L^2[0,T;L^2(\Omega)]} \leq \frac{C_{\infty}}{\epsilon^2} \|e_h\|_{L^2[0,T;L^2(\Omega)]}.
\end{eqnarray*}
Here $\Delta_h \psi_h \in {\mathcal U}_h$ denotes a discrete approximation of $\Delta \psi$, defined by, $ (\Delta_h \psi_h(.),w_h) =  a(\psi_h(.),w_h)$, for all $w_h \in U_h$ and for every $t \in (t^{n-1},t^n]$.  
\end{lem}
\begin{proof} {\it Step 1: Stability estimates in $L^{\infty}[0,T;L^2(\Omega)] \cap L^2[0,T;H^1(\Omega)]$}:
We rewrite (\ref{eqn:orth4}) as follows: 
\begin{eqnarray} \label{eqn:orth4aux}
&& -({\psi}^n_{h+},w^n_{h-}) + \int_{t^{n-1}}^{t^n} \big ( (\psi_h,w_{ht})
+ \epsilon^2 a({\psi}_h,w_h)  \big )dt  +({\psi}^{n-1}_{h+},w^{n-1}_{h+}) \nonumber \nonumber \\
&& + (1-\epsilon^2) \left ( \int_{t^{n-1}}^{t^n} a({\psi}_h,w_h) dt + \frac{1}{\epsilon^2} \int_{t^{n-1}}^{t^n} \left ( 3(u^2-1) \psi_h,w_h \right ) dt \right ) \nonumber \\
&& + \int_{t^{n-1}}^{t^n} ( (3u^2-1) \psi_h,w_h ) dt  = \int_{t^{n-1}}^{t^n} (e_{h},w_h)dt.
\end{eqnarray}
Setting $w_h = \psi_h$ into (\ref{eqn:orth4aux}) and using the spectral estimate (\ref{eqn:spec}) we deduce,
\begin{eqnarray*}
&& \frac{1}{2} \|\psi^{n-1}_{h+}\|^2_{L^2(\Omega)} - \frac{1}{2} \|\psi^{n}_{h+}\|^2_{L^2(\Omega)} + \frac{1}{2} \|[\psi^n_h]\|^2_{L^2(\Omega)} + C\epsilon^2 \int_{t^{n-1}}^{t^n} \|\psi_h\|^2_{H^1(\Omega)} dt \\
&& \qquad - (1-\epsilon^2) C_s \int_{t^{n-1}}^{t^n} \|\psi_h\|^2_{L^2(\Omega)}dt + 3 \int_{t^{n-1}}^{t^n} \|u\psi_h\|^2_{L^2(\Omega)} dt \\
&& \leq \frac{3}{2} \int_{t^{n-1}}^{t^n} \|\psi_h\|^2_{L^2(\Omega)} dt  + \frac{1}{2} \int_{t^{n-1}}^{t^n} \|e_h\|^2_{L^2(\Omega)} dt.
\end{eqnarray*}
Hence, using standard algebra we obtain, 
\begin{eqnarray} \label{eqn:psiaux}
&& \frac{1}{2} \|\psi^{n-1}_{h+}\|^2_{L^2(\Omega)} - \frac{1}{2} \|\psi^{n}_{h+}\|^2_{L^2(\Omega)} + \frac{1}{2} \|[\psi^n_h]\|^2_{L^2(\Omega)} + C\epsilon^2 \int_{t^{n-1}}^{t^n} \|\psi_h\|^2_{H^1(\Omega)} dt \nonumber \\
&&  + 3 \int_{t^{n-1}}^{t^n} \|u \psi_h\|^2_{L^2(\Omega)} dt \leq C(C_s) \int_{t^{n-1}}^{t^n} \|\psi_h\|^2_{L^2(\Omega)} dt  + \frac{1}{2} \int_{t^{n-1}}^{t^n} \|e_h\|^2_{L^2(\Omega)} dt.
\end{eqnarray}
where the constant $C(C_s)$ depends on $C_s$ but it is independent of $\epsilon$.
For low order schemes $k=0,1$, a standard Gronwall Lemma provides the estimates at arbitrary time points, as well as the estimate in $L^2[0,T;H^1(\Omega)]$.
For higher order schemes, we proceed using the technique of Section 4, based on the approximation of the discrete characteristic. Hence, following exactly the same approach as in Proposition 4.3, for fixed $t \in (t^{n-1},t^n)$, we obtain with $z_h \in U_h$ independent of $t$, and $\rho$ defined as in Lemma \ref{lem:char1} (suitably modified to handle the backwards in time problem)
\begin{eqnarray*}
&& (\psi_{h}(t)-\psi^{n}_{h+},z_h) \\
&& = - \int_{t^{n-1}}^{t^n} \left (
a(\psi_h,z_h \rho) + (1/\epsilon^2) ((3u^2-1) \psi_h,z_h \rho) \right ) ds
+ \int_{t^{n-1}}^{t^n} (e_h,z_h \rho ) ds \\
&& \leq  C_k \Big [ \int_{t^{n-1}}^{t^n} \|\nabla
\psi_h\|_{L^2(\Omega)} \|\nabla z_h\|_{L^2(\Omega)} + \int_{t^{n-1}}^{t^n} \|e_h\|_{L^2(\Omega)}
\|z_h\|_{L^2(\Omega)} ds \\
&& + (1/\epsilon^2) \int_{t^{n-1}}^{t^n} \left (
\|u\psi_h\|_{L^{2}(\Omega)} \|u\|_{L^6(\Omega)} \|z_h\|_{L^3(\Omega)} +
\|\psi_h\|_{L^2(\Omega)} \|z_h\|_{L^2(\Omega)}  \right ) ds  \Big ] ,
\end{eqnarray*} 
where $C_k$ is the constant of Lemma \ref{lem:char1}. Since, $z_h \in U_h$ is independent of $t$, we deduce, 
\begin{eqnarray*}
&&\hskip-20pt (\psi_{h}(t)-\phi^{n}_{h+},z_h) \leq  C_k \Big [ \|z_h\|_{H^1(\Omega)} \int_{t^{n-1}}^{t^n}
\|\psi_h\|_{H^1(\Omega)}ds  +  \|z_h\|_{L^2(\Omega)}
\int_{t^{n-1}}^{t^n} \|e_h\|_{L^2(\Omega)}  ds  
\\
&&\hskip-20pt +\|u\|_{L^{\infty}[0,T;L^6(\Omega)]} \frac{1}{\epsilon^2} \|z_h\|_{L^3(\Omega)}    \int_{t^{n-1}}^{t^n}
\|u \psi_h\|_{L^{2}(\Omega)} ds + \frac{1}{\epsilon^2} \|z_h\|_{L^2(\Omega)}\int_{t^{n-1}}^{t^n} \|\psi_h\|_{L^2(\Omega)} ds \Big ]. 
\end{eqnarray*}
Therefore, setting $z_h = \psi_h(t)$ (for the previously fixed $t$), integrating with respect to time, using the inequality $\|.\|_{L^3(\Omega)} \leq C\|.\|^{1/2}_{L^2(\Omega)} \|.\|^{1/2}_{H^1(\Omega)}$ and using H\"older's and Young's inequalities we deduce (with different $C_k$),
\begin{eqnarray*}
&&\hskip-15pt \int_{t^{n-1}}^{t^n} \|\psi_h(t)\|^2_{L^2(\Omega)} dt \leq C_k \tau_n \|\psi^{n}_{h+}\|^2_{L^2(\Omega)} + C_k \tau_n \|\psi_h\|^2_{L^2[t^{n-1},t^n;H^1(\Omega)]} \\
&&\hskip-15pt  + \|u\|_{L^{\infty}[0,T;L^6(\Omega)]}\frac{C_k \tau_n}{\epsilon^2} \|u\psi_h\|_{L^2[t^{n-1},t^n;L^2(\Omega)]} \|\psi_h\|^{1/2}_{L^2[t^{n-1},t^n;L^2(\Omega)]} \|\psi_h\|^{1/2}_{L^2[t^{n-1},t^n;H^1(\Omega)]}  \\
&& \hskip-15pt + \frac{C_k}{\epsilon^2} \tau_n \|\psi_h\|_{L^2[t^{n-1},t^n;L^2(\Omega)]} \|e_h\|_{L^2[t^{n-1},t^n;L^2(\Omega)]}. \\
&& \hskip-15pt  \leq C_k \tau_n  \|\psi_h\|^2_{L^2[t^{n-1},t^n;H^1(\Omega)]} + \|u\|_{L^{\infty}[0,T;L^6(\Omega)]} \frac{C_k\tau_n}{\epsilon^3} \|u\psi_h\|^2_{L^2[t^{n-1},t^n;L^2(\Omega)]} \\
&& +  \|u\|_{L^{\infty}[0,T;L^6(\Omega)]} \frac{C_k\tau_n }{\epsilon}  \|\psi_h\|_{L^2[t^{n-1},t^n;L^2(\Omega)]} \|\psi_h\|_{L^2[t^{n-1},t^n;H^1(\Omega)]} \\
&& + \frac{C_k \tau_n}{\epsilon^2} \|\psi_h\|_{L^2[t^{n-1},t^n;L^2(\Omega)]} \|e_h\|_{L^2[t^{n-1},t^n;L^2(\Omega)]}.
\end{eqnarray*}
The proof is now completed using standard techniques. Indeed, we choose $\tau_n$ small enough to hide the $L^2[t^{n-1},t^n;L^2(\Omega)]$ on the left and then we substitute the resulting bound into (\ref{eqn:psiaux}) and we hide the terms involving $\|u\psi_h\|_{L^2[t^{n-1},t^n;L^2(\Omega)]}$ and $\|.\|_{L^2[0,T;H^1(\Omega)]}$ on the left. \\
{\it Step 2: Stability estimates in $L^{\infty}[0,T;H^1(\Omega)]$}: The proof is essentially contained in \cite[Theorem 4.10]{ChWa10}. For completeness we describe the main arguments. By definition of $\Delta_h \psi_h$, and since $\psi_h \in {\mathcal P}_k[t^{n-1},t^n;U_h]$, we also have that
$\Delta_h \psi_h \in {\mathcal P}_k[t^{n-1},t^n;U_h]$. Setting $w_h = \psi_{ht}$, and $w_h = \Delta_h \psi_h$ we deduce,
$$ \frac{1}{2} \|\nabla \psi_h\|^2_{L^2(\Omega)} = ( \Delta_h \psi_h, \psi_{ht} ), \qquad\mbox{ and} \qquad a( \psi_h, \Delta_h \psi_h ) \equiv \|\Delta_h \psi_h \|^2_{L^2(\Omega)}. $$
Hence, setting $\Delta_h \psi_h$ into (\ref{eqn:orth4}), substituting the last two equalities and using standard algebra we obtain,
\begin{eqnarray} \label{eqn:psiinfty}
&& (1/2) \|\nabla \psi^{n-1}_{h+}\|^2_{L^2(\Omega)} + (1/2) \|[\nabla \psi^{n-1}_h]\|^2_{L^2(\Omega)} + \int_{t^{n-1}}^{t^n} \|\Delta_h \psi_h\|^2_{L^2(\Omega)} dt \nonumber \\
&& + \frac{1}{\epsilon^2} \int_{t^{n-1}}^{t^n} (3u^2 \psi_h, \Delta_h \psi_h) dt - \frac{1}{\epsilon^2} \int_{t^{n-1}}^{t^n} ( \psi_h,\Delta_h \psi_h) dt \nonumber \\
&& = (1/2) \|\nabla \psi^{n}_{h+1}\|^2_{L^2(\Omega)} + \int_{t^{n-1}}^{t^n} ( e_h, \Delta_h \psi_h) dt. 
\end{eqnarray}
Note that 
\begin{eqnarray*}
&& \Big |  \frac{1}{\epsilon^2} \int_{t^{n-1}}^{t^n} (3u^2 \psi_h, \Delta_h \psi_h) - (\psi_h,\Delta_h \psi_h) dt \Big | \\
&& \leq (1/4) \int_{t^{n-1}}^{t^n} \|\Delta_h \psi_h\|^2_{L^2(\Omega)} dt + \frac{C^2_{\infty}}{\epsilon^4} \int_{t^{n-1}}^{t^n} \|u\psi_h\|^2_{L^2(\Omega)} dt + \frac{1}{\epsilon^4} \int_{t^{n-1}}^{t^n} \|\psi_h\|^2_{L^2(\Omega)} dt.
\end{eqnarray*}
Substituting the above inequality into (\ref{eqn:psiinfty}) and summing the resulting inequalities, and using the bounds $\|u\psi_h\|_{L^2[0,T;L^2(\Omega)} \leq C \|e_h\|_{L^2[0,T;L^2(\Omega)]}$ and $\|\psi_h\|_{L^2[0,T;L^2(\Omega)]} \leq C \|e_h\|_{L^2[0,T;L^2(\Omega)]}$, we deduce that,
\begin{eqnarray*} 
&& \|\psi^{n-1}_{h+}\|^2_{L^2(\Omega)} + \|\Delta_h \psi_h\|^2_{L^2[0,T;L^2(\Omega)]} \leq \frac{C^2_{\infty}+1}{\epsilon^4} \|e_h\|^2_{L^2[0,T;L^2(\Omega)]},
\end{eqnarray*}
which is the desired estimate. 
The stability bound in $L^{\infty}[0,T;H^1(\Omega)]$ follows directly from the above technique when $k=0,1$. For higher order schemes we refer the reader to 
\cite[Theorem 4.10]{ChWa10}.

\end{proof}


Now, we are ready to prove the following bound, which will allow us to apply a bootstrap argument. Using an appropriate duality argument, we avoid the use of Gr\"onwall type inequalities. 
\begin{prop} \label{prop:error}
Let $\tau,h,\epsilon$, satisfy $\tau \leq {\tilde \tau}$, $h \leq \tilde h$ (where $\tilde \tau, \tilde h$ defined in (\ref{eqn:conv})) and the assumptions of Lemma \ref{lem:auxpsi}. Suppose also that $\tau,h$ satisfy  
\begin{enumerate} 
\item $\tau+h \leq \frac{\delta C\epsilon^{4}}{(\|u\|_{L^2[0,T;H^2(\Omega)]} + \|u_t\|_{L^2[0,T;H^1(\Omega)]})}$, when $d=3$, 
\item $\tau+h \leq \frac{\delta C\epsilon^{7/2}}{(\|u\|_{L^2[0,T;H^2(\Omega)]} + \|u_t\|_{L^2[0,T;H^1(\Omega)]})}$, when $d=2$
\item $\tau +h \leq \frac{\delta C \epsilon^{3}}{(\|u\|_{L^2[0,T;H^2(\Omega)]} + \|u_t\|_{L^2[0,T;H^1(\Omega)]})}$, when $d=2$, $k=0,1$.
\end{enumerate}
where $\delta >0$ (to be chosen later)
with constant $C$ depending only upon the domain (independent of $\epsilon,h,\tau$).
Then, there exists a constant  $C>0$ independent of $\tau,h,\epsilon$, such that following estimate hold: 
\begin{align*}
& \|e_h\|_{L^2[0,T;L^2(\Omega)]} \leq C \Big ( \frac{1}{\epsilon^2} (\|u_p\|^2_{L^{\infty}[0,T;L^6(\Omega)]} + \|u\|^2_{L^{\infty}[0,T;L^6(\Omega)]} ) \|e_p\|_{L^2[0,T;H^1(\Omega)]} \\
& + \frac{1}{\epsilon^2} \|e_p\|_{L^2[0,T;L^2(\Omega)]} + \|u\|_{L^{\infty}[0,T;L^{\infty}(\Omega)]} \|e_p\|_{L^2[0,T;L^2(\Omega)]} \\
& + C\delta (\|e_hu_h\|_{L^2[0,T;L^2(\Omega)]} + \|e_hu_p\|_{L^2[0,T;L^2(\Omega)]}) \Big ).
\end{align*}
\end{prop}
\begin{proof}
Setting $w_h = e_h$ into (\ref{eqn:orth4}), and using integration by parts in time, we obtain: For all $n=N,...,1$, 
\begin{eqnarray} \label{eqn:orth5}
&&  -({\psi}^n_{h+},e^n_{h-}) +  ({\psi}^{n}_{h-},e^{n}_{h-}) + \int_{t^{n-1}}^{t^n}  -(\psi_{ht},e_h) dt + \int_{t^{n-1}}^{t^n} \big ( a({\psi}_{h},e_h) dt  \nonumber \\
&& + \int_{t^{n-1}}^{t^n} \left ( (3u^2-1) \psi_h, e_h \right ) dt = \int_{t^{n-1}}^{t^n} \|e_h\|^2_{L^2(\Omega)}dt  
\end{eqnarray}
Setting $w_h = \psi_{h}$ into (\ref{eqn:orth3}), we deduce for all $n=1,...,N$,  
\begin{eqnarray}
&& (e^n_{h-},\psi^n_{h-}) + \int_{t^{n-1}}^{t^n}
        \left (  - ( e_h,\psi_{ht} ) +  a(e_h,\psi_h) \right ) dt  -  (e^{n-1}_{h-},\psi^{n-1}_{h+}) \nonumber \\
&&  +   \frac{1}{\epsilon^2}  \int_{t^{n-1}}^{t^n} \left ( (e_h(u^2_h+u^2_p+u_hu_p),\psi_h)-(e_h,\psi_h) \right  ) dt  \nonumber \\
 && =      \frac{1}{\epsilon^2}  \int_{t^{n-1}}^{t^n} \left ( (e_p(u^2_p+u^2+u_pu),\psi_h)-(e_p,\psi_h) \right  ) dt .  \label{eqn:orth6} 
\end{eqnarray}
Subtracting (\ref{eqn:orth6}) from (\ref{eqn:orth5}), and rearranging terms, we obtain, for all $n=1,...,N$,
\begin{eqnarray*}   
&& \int_{t^{n-1}}^{t^n} \|e_h\|^2_{L^2(\Omega)}dt =  -({\psi}^n_{h+},e^n_{h-}) +  (e^{n-1}_{h-},\psi^{n-1}_{h+}) \nonumber \\
&& \qquad +   \frac{1}{\epsilon^2}  \int_{t^{n-1}}^{t^n} \left ( (e_p(u^2_p+u^2+u_pu),\psi_h)-(e_p,\psi_h) \right  ) dt \nonumber \\
&& \qquad +  \frac{1}{ \epsilon^2} \int_{t_{n-1}}^{t_n} \left ( (u^2_h+u^2_p+u_hu_p-3u^2)e_h,\psi_h \right )  dt. 
\end{eqnarray*}
or equivalently 
\begin{eqnarray} \label{eqn:orth6b}  
&& \int_{t^{n-1}}^{t^n} \|e_h\|^2_{L^2(\Omega)}dt =  -({\psi}^n_{h+},e^n_{h-}) +  (e^{n-1}_{h-},\psi^{n-1}_{h+}) \nonumber \\
&& \qquad +   \frac{1}{\epsilon^2}  \int_{t^{n-1}}^{t^n} \left ( (e_p(u^2_p+u^2+u_pu),\psi_h)-(e_p,\psi_h) \right  ) dt \nonumber \\
&& \qquad +  \frac{1}{ \epsilon^2} \int_{t_{n-1}}^{t_n} \left (  ((u^2_h-u^2)+(u^2_p-u^2)+u_hu_p-u^2)e_h,\psi_h \right )  dt. 
\end{eqnarray}
First, note adding and subtracting $u^2_p$ in the term $u^2_h-u^2$, using the relation,
$$u_hu_p-u^2= (u_h-u_p+u_p)u_p - u^2 \equiv (u_h-u_p)u_p+ u^2_p-u^2$$
and substituting the resulting relation into (\ref{eqn:orth6b}) we arrive at:
\begin{eqnarray} \label{eqn:orth6c} 
&& \int_{t^{n-1}}^{t^n} \|e_h\|^2_{L^2(\Omega)}dt =  -({\psi}^n_{h+},e^n_{h-}) +  (e^{n-1}_{h-},\psi^{n-1}_{h+}) \nonumber \\
&& \qquad +   \frac{1}{\epsilon^2}  \int_{t^{n-1}}^{t^n} \left ( (e_p(u^2_p+u^2+u_pu),\psi_h)-(e_p,\psi_h) \right  ) dt \nonumber \\
&& \qquad +  \frac{1}{ \epsilon^2} \int_{t_{n-1}}^{t_n} \left ( \left ( (u^2_h-u^2_p)+3(u^2_p-u^2) + (u_h-u_p)u_p \right ) e_h,\psi_h \right ) dt. 
\end{eqnarray}
Summing the equalities (\ref{eqn:orth6c}), noting that $e^0_{h-} = 0 = \phi^N_{h+}$, and using H\"older's and Young's inequalities, and the identity $a^2-b^2=(a-b)(a+b)$, we obtain,
\begin{eqnarray}  \label{eqn:orth7}
&& \int_0^T \|e_h\|^2_{L^2(\Omega)} dt \leq  \frac{C}{\epsilon^2} \int_{0}^{T} \|e_p\|_{L^6(\Omega)} (\|u^2_p\|_{L^3(\Omega)} + \|u^2\|_{L^3(\Omega)} )\|\psi_h\|_{L^2(\Omega)} dt  
 \nonumber \\
&&  \qquad+ \frac{1}{\epsilon^2} \int_0^T \|e_p\|_{L^2(\Omega)} \|\psi_h\|_{L^2(\Omega)} dt \nonumber \\
&& \qquad + \frac{2}{\epsilon^2} \int_0^T (\|e_hu_h\|_{L^2(\Omega)} + \|e_hu_p\|_{L^2(\Omega)} ) \|e_h\|_{L^{2}(\Omega)} \|\psi_h\|_{L^{\infty}(\Omega)} dt \nonumber \\
&& \qquad + \frac{3}{\epsilon^2} \int_0^T \|e_p\|_{L^2(\Omega)} 
\|e_hu_p\|_{L^2(\Omega)} \|\psi_h\|_{L^{\infty}(\Omega)} dt \nonumber \\
&& \qquad + \frac{3}{\epsilon^2} \int_0^T \|u\|_{L^{\infty}(\Omega)} \|e_p\|_{L^2(\Omega)} \|e_h\|_{L^2(\Omega)} \|\psi_h\|_{L^{\infty}(\Omega)} dt.
\end{eqnarray}
For $d=3$, we employ the inequality $\|\psi_h\|_{L^{\infty}(\Omega)} \leq C \|\nabla \psi_h\|^{1/2}_{L^2(\Omega)} \|\Delta_h \psi_h\|^{1/2}_{L^2(\Omega)}$ (see e.g. \cite[pp 298]{HeRa82} to get
\begin{eqnarray*}  \
&& \int_0^T \|e_h\|^2_{L^2(\Omega)} dt \\
&& \leq  \frac{C}{\epsilon^2} \int_{0}^{T} \|e_p\|_{L^6(\Omega)} (\|u^2_p\|_{L^3(\Omega)} + \|u^2\|_{L^3(\Omega)} )\|\psi_h\|_{L^2(\Omega)} dt  
 \nonumber \\
&&  + \frac{1}{\epsilon^2} \int_0^T \|e_p\|_{L^2(\Omega)} \|\psi_h\|_{L^2(\Omega)} dt \nonumber \\
&& + \frac{2}{\epsilon^2} \|\psi_h\|^{1/2}_{L^{\infty}[0,T;H^1(\Omega)]}  \int_0^T (\|e_hu_h\|_{L^2(\Omega)} + \|e_hu_p\|_{L^2(\Omega)} ) \|e_h\|_{L^{2}(\Omega)} \|\Delta_h \psi_h\|^{1/2}_{L^{2}(\Omega)} dt \nonumber \\
&& + \frac{3}{\epsilon^2} \|\psi_h\|^{1/2}_{L^{\infty}[0,T;H^1(\Omega)]} \int_0^T \|e_p\|_{L^2(\Omega)} 
\|e_hu_p\|_{L^2(\Omega)} \|\Delta_h \psi_h\|^{1/2}_{L^{2}(\Omega)} dt \nonumber \\
&& + \frac{3}{\epsilon^2} \|\psi_h\|^{1/2}_{L^{\infty}[0,T;H^1(\Omega)]} \int_0^T \|u\|_{L^{\infty}(\Omega)} \|e_p\|_{L^2(\Omega)} \|e_h\|_{L^2(\Omega)} \|\Delta_h \psi_h\|^{1/2}_{L^{2}(\Omega)} dt. \nonumber 
\end{eqnarray*}

Therefore, using the stability bounds of $\psi_h$ of Lemma \ref{lem:auxpsi}, i.e., 
$\|\psi_h\|_{L^2[0,T;L^2(\Omega)]} \leq \|e_h\|_{L^2[0,T;L^2(\Omega)]}$, $\|\psi_h\|_{L^{\infty}[0,T;L^6(\Omega)]} \leq \frac{C}{\epsilon^2} \|e_h\|_{L^2[0,T;L^2(\Omega)]}$, and $\|\Delta_h \psi_h\|_{L^2[0,T;L^2(\Omega)]} \leq \frac{C}{\epsilon^2} \|e_h\|_{L^2[0,T;L^2(\Omega)]}$  to deduce, 
\begin{eqnarray*} \label{eqn:orth7b}
&&\int_0^T \|e_h\|^2_{L^2(\Omega)} dt \nonumber \\
&& \leq \frac{C}{\epsilon^2}  (\|u_p\|^2_{L^{\infty}[0,T;L^6(\Omega)]} 
+ \|u\|^2_{L^{\infty}[0,T;L^6(\Omega)]} ) \|e_p\|_{L^2[0,T;L^6(\Omega)]} \|e_h\|_{L^2[0,T;L^2(\Omega)]}\nonumber \\
&& + \frac{1}{\epsilon^2} \|e_p\|_{L^2[0,T;L^2(\Omega)]} \|e_h\|_{L^2[0,T;L^2(\Omega)]} \nonumber \\
&& +\frac{C}{\epsilon^4} (\|e_hu_h\|_{L^2[0,T;L^2(\Omega)]}+ \|e_hu_p\|_{L^2[0,T;L^2(\Omega)]})
 \|e_h\|_{L^4[0,T;L^2(\Omega)]} \|e_h\|_{L^2[0,T;L^2(\Omega)]} \nonumber \\
&& + \frac{C}{\epsilon^4} \|e_p\|_{L^4[0,T;L^2(\Omega)]} \|e_hu_p\|_{L^2[0,T;L^2(\Omega)]} \|e_h\|_{L^2[0,T;L^2(\Omega)]} \nonumber \\
&& +\frac{C}{\epsilon^4} \|u\|_{L^{\infty}[0,T;L^{\infty}(\Omega)]}  \|e_p\|_{L^2[0,T;L^2(\Omega)]} \|e_h\|_{L^4[0,T;L^2(\Omega)]} \|e_h\|_{L^2[0,T;L^2(\Omega)]}. 
\end{eqnarray*}
Note due to the Theorem \ref{thm:conv} there exists $\tilde \tau$, $\tilde h$ such that $\|e\|_{L^4[0,T;L^2(\Omega)]} \leq \delta \epsilon^{4}$ for every $\tau \leq \tilde \tau$ and $h \leq \tilde h$. Hence, using (\ref{eqn:projest1}), and the improved estimate in $L^2[0,T;L^2(\Omega)]$, (see e.g. \cite{MeVe08}),
we obtain that
\begin{eqnarray*}
&& \|e_h\|_{L^4[0,T;L^2(\Omega)]} \leq \delta \epsilon^4 + \|e_p\|_{L^4[0,T;L^2(\Omega)]}. \\
&& \leq \delta \epsilon^4 +  \|e_p\|^{1/2}_{L^{\infty}[0,T;L^2(\Omega)]} \|e_p\|^{1/2}_{L^2[0,T;L^2(\Omega)]} \\
&& \leq \delta \epsilon^4 + C(\tau+h)^{1/2} (\tau+h^2)^{1/2}  (\|u\|_{L^2[0,T;H^2(\Omega)]} + \|u_t\|_{L^2[0,T;H^1(\Omega)]}) \leq 2\delta \epsilon^4, 
\end{eqnarray*}
provided that $\tau+h \leq \frac{\delta C\epsilon^{4}}{(\|u\|_{L^2[0,T;H^2(\Omega)]} + \|u_t\|_{L^2[0,T;H^1(\Omega)]})}$. 
Substituting the above bound, we deduce,
\begin{eqnarray*} \label{eqn:orth7c}
&& \int_0^T \|e_h\|^2_{L^2(\Omega)} dt  \\
&& \leq \frac{C}{\epsilon^2} (\|u_p\|^2_{L^{\infty}[0,T;L^6(\Omega)]} 
+ \|u\|^2_{L^{\infty}[0,T;L^6(\Omega)]} ) \|e_p\|_{L^2[0,T;L^6(\Omega)]}
\|e_h\|_{L^2[0,T;L^2(\Omega)]} \nonumber \\
&& + \frac{1}{\epsilon^2} \|e_p\|_{L^2[0,T;L^2(\Omega)]} \|e_h\|_{L^2[0,T;L^2(\Omega)]} \nonumber \\
&& + \delta (\|e_hu_h\|_{L^2[0,T;L^2(\Omega)]}+ \|e_hu_p\|_{L^2[0,T;L^2(\Omega)]})
 \|e_h\|_{L^2[0,T;L^2(\Omega)]} \nonumber \\
&& + \|u\|_{L^{\infty}[0,T;L^{\infty}(\Omega)]} \|e_p\|_{L^2[0,T;L^2(\Omega)]}  \|e_h\|_{L^2[0,T;L^2(\Omega)]} \nonumber . 
\end{eqnarray*}
The estimate for the three dimensional case now follows by standard algebra. 
For $d=2$, we note that $\|\psi_h\|_{L^{\infty}(\Omega)} \leq C \|\psi_h\|^{1/2}_{L^2(\Omega)} \|\Delta_h \psi_h\|^{1/2}_{L^2(\Omega)}$ (see \cite{HeRa82}), hence using the stability bounds of Lemma \ref{lem:auxpsi}, and in particular the fact that $\|\psi_h\|_{L^{\infty}[0,T;L^2(\Omega)]} \leq \frac{C}{\epsilon} \|e_h\|_{L^2[0,T;L^2(\Omega)]}$, we deduce from (\ref{eqn:orth7}), 
\begin{eqnarray*} \label{eqn:orth7b}
&&\int_0^T \|e_h\|^2_{L^2(\Omega)} dt \nonumber \\
&& \leq \frac{C}{\epsilon^2}  (\|u_p\|^2_{L^{\infty}[0,T;L^6(\Omega)]} 
+ \|u\|^2_{L^{\infty}[0,T;L^6(\Omega)]} ) \|e_p\|_{L^2[0,T;L^6(\Omega)]} \|e_h\|_{L^2[0,T;L^2(\Omega)]}\nonumber \\
&& + \frac{1}{\epsilon^2} \|e_p\|_{L^2[0,T;L^2(\Omega)]} \|e_h\|_{L^2[0,T;L^2(\Omega)]} \nonumber \\
&& +\frac{C}{\epsilon^{7/2}}  (\|e_hu_h\|_{L^2[0,T;L^2(\Omega)]}+ \|e_hu_p\|_{L^2[0,T;L^2(\Omega)]})
 \|e_h\|_{L^4[0,T;L^2(\Omega)]} \|e_h\|_{L^2[0,T;L^2(\Omega)]} \nonumber \\
&& + \frac{C}{\epsilon^{7/2}}  \|e_p\|_{L^4[0,T;L^2(\Omega)]} \|e_hu_p\|_{L^2[0,T;L^2(\Omega)]}  \|e_h\|_{L^2[0,T;L^2(\Omega)]} \nonumber \\
&& +\frac{C}{\epsilon^{7/2}} \|u\|_{L^{\infty}[0,T;L^{\infty}(\Omega)]} \|e_p\|_{L^2[0,T;L^2(\Omega)]} \|e_h\|_{L^4[0,T;L^2(\Omega)]} \|e_h\|_{L^2[0,T;L^2(\Omega)]}. 
\end{eqnarray*}
The proof now follows using similar arguments. Indeed, choosing $\tilde \tau, \tilde h$ to guarantee, $\|e\|_{L^4[0,T;L^2(\Omega)]} \leq \delta \epsilon^{7/2}$, for $\tau \leq \tilde \tau$, $h \leq \tilde h$, and noting that 
\begin{eqnarray*} \|e_p\|_{L^4[0,T;L^2(\Omega)]} \leq C (\tau + h) ( \|u\|_{L^2[0,T;H^2(\Omega)]} + \|u_t\|_{L^2[0,T;H^1(\Omega)]} \leq C \delta \epsilon^{7/2}
\end{eqnarray*}
provided that $\tau + h \leq \frac{C\delta \epsilon^{7/2}} {  \|u\|_{L^2[0,T;H^2(\Omega)]} + \|u_t\|_{L^2[0,T;H^1(\Omega)]}}$, we derive the desired estimate.
Finally, we turn our attention to the case where $k=0,1$ and $d=2$. Then, we note that Lemma \ref{lem:auxpsi}, implies that 
$\|\psi_h \|_{L^{\infty}[0,T;L^2(\Omega)]} \leq C$, where $C$ is independent of $\epsilon,\tau,h$. As a consequense, we deduce from (\ref{eqn:orth7}), 
\begin{eqnarray*} \label{eqn:orth7b}
&&\int_0^T \|e_h\|^2_{L^2(\Omega)} dt \nonumber \\
&& \leq \frac{C}{\epsilon^2}  (\|u_p\|^2_{L^{\infty}[0,T;L^6(\Omega)]} 
+ \|u\|^2_{L^{\infty}[0,T;L^6(\Omega)]} ) \|e_p\|_{L^2[0,T;L^6(\Omega)]} \|e_h\|_{L^2[0,T;L^2(\Omega)]}\nonumber \\
&& + \frac{1}{\epsilon^2} \|e_p\|_{L^2[0,T;L^2(\Omega)]} \|e_h\|_{L^2[0,T;L^2(\Omega)]} \nonumber \\
&& +\frac{C}{\epsilon^{3}}  (\|e_hu_h\|_{L^2[0,T;L^2(\Omega)]}+ \|e_hu_p\|_{L^2[0,T;L^2(\Omega)]})
 \|e_h\|_{L^4[0,T;L^2(\Omega)]} \|e_h\|_{L^2[0,T;L^2(\Omega)]} \nonumber \\
&& + \frac{C}{\epsilon^{3}}  \|e_p\|_{L^4[0,T;L^2(\Omega)]} \|e_hu_p\|_{L^2[0,T;L^2(\Omega)]}  \|e_h\|_{L^2[0,T;L^2(\Omega)]} \nonumber \\
&& +\frac{C}{\epsilon^{3}} \|u\|_{L^{\infty}[0,T;L^{\infty}(\Omega)]} \|e_p\|_{L^2[0,T;L^2(\Omega)]} \|e_h\|_{L^4[0,T;L^2(\Omega)]} \|e_h\|_{L^2[0,T;L^2(\Omega)]}. 
\end{eqnarray*}
Therefore, we derive the desired estimate, provided that  $\tilde \tau, \tilde h$ are chosen to guarantee, $\|e\|_{L^4[0,T;L^2(\Omega)]} \leq \delta \epsilon^{3}$, for $\tau \leq \tilde \tau$, $h \leq \tilde h$, and  $\tau + h \leq \frac{C\delta \epsilon^{3}} {  \|u\|_{L^2[0,T;H^2(\Omega)]} + \|u_t\|_{L^2[0,T;H^1(\Omega)]}}$
\end{proof}
\begin{rmk}
There are many ways to write the additional time step and spatial size restrictions.  
The assumptions (1) and (2) of Proposition 5.5 can be replaced  by the more general assumption $\|e_p\|_{L^4[0,T;L^2(\Omega)]} \leq C \delta \epsilon^4$. Since $e_p= u_p-u$ refers to the standard error related to discontinuous Galerkin approximation of a linear parabolic pde, with right hand side $u_t-\Delta u$. Therefore, from \ref{eqn:projest1} and Lemma \ref{lem:rates1}, for instance. we may derive the following restriction when $d=3$
\begin{eqnarray*}
&& \|e_p\|_{L^4[0,T;L^2(\Omega)]} \leq C\left ( \frac{\tau^{k+1}}{h} \|u^{(k+1)}\|_{L^2[0,T;L^2(\Omega)]} + h^l \|u\|_{L^2[0,T;H^{l+1}(\Omega)]} \right )^{1/2} \\
&& \qquad \times \left ( \tau^{k+1} \|u^{(k+1)} \|_{L^2[0,T;L^2(\Omega)]} + h^{l+1} \|u\|_{L^2[0,T;H^{l+1}(\Omega)]} \right )^{1/2} \leq C \delta \epsilon^4.
\end{eqnarray*}
Similarly, the proof is still valid even when limited regularity assumptions are present. Indeed, even when $u \in L^2[0,T;H^2(\Omega)] \cap H^1[0,T;L^2(\Omega)]$ regularity is available then using the bound $\|e_p\|_{L^4[0,T;L^2(\Omega)]} \leq C \|e_p\|^{1/2}_{L^{\infty}[0,T;L^2(\Omega)]} \|e_p\|^{1/2}_{L^2[0,T;L^2(\Omega)]}$, we deduce the restrictions, 
\begin{enumerate}
\item  $(\tau^{1/2} + h)^{3/2} \leq \frac{ \delta C \epsilon^{4}}{(\|u\|_{L^2[0,T;H^2(\Omega)]} + \|u_t\|_{L^2[0,T;L^2(\Omega)]})} $ when $d=3$,  
\item $ (\tau^{1/2} + h)^{3/2}  \leq \frac{  \delta C \epsilon^{7/2}}{(\|u\|_{L^2[0,T;H^2(\Omega)]} + \|u_t\|_{L^2[0,T;L^2(\Omega)]})}$ when $d=2$.
\item $(\tau^{1/2} + h)^{3/2} \leq \frac{ \delta \epsilon^3}{(\|u\|_{L^2[0,T;H^2(\Omega)]} + \|u_t\|_{L^2[0,T;L^2(\Omega)]})}$ when $d=2$, $k=0,1$.
\end{enumerate}
\end{rmk}
\subsection{Best approximation error estimates} Now, we are ready to proceed with the main estimate, using a boot-strap argument.
\begin{thm} \label{thm:est} 
Let $\tau,h,\epsilon$, satisfy $\tau \leq {\tilde \tau}$, $h \leq \tilde h$ (where $\tilde \tau, \tilde h$ defined in (\ref{eqn:conv})) and the assumptions of Lemma \ref{lem:auxpsi}. Suppose also that $\tau,h$ satisfy  
\begin{enumerate} 
\item $\tau +h \leq \frac{\delta C\epsilon^4}{(\|u\|_{L^2[0,T;H^2(\Omega)]} + \|u_t\|_{L^2[0,T;H^1(\Omega)]})}$, when $d=3$, 
\item $\tau +h \leq \frac{\delta C\epsilon^{7/2}}{(\|u\|_{L^2[0,T;H^2(\Omega)]} + \|u_t\|_{L^2[0,T;H^1(\Omega)]})}$, when $d=2$.
\item $\tau + h \leq \frac{ \delta C \epsilon^3}{(\|u\|_{L^2[0,T;H^2(\Omega)]} + \|u_t\|_{L^2[0,T;H^1(\Omega)]})}$ when $d=2$, $k=0,1$.
\end{enumerate}
Then, there exists a constant (still) denoted by $C$ depending only upon $\Omega$, and $C_k$ but independent of $\epsilon$, such that, 
\begin{eqnarray*}
&& \|e_h\|^2_{L^2[0,T;H^1(\Omega)]}  + (1/\epsilon^2) ( \|e_hu_h\|^2_{L^2[0,T;L^2(\Omega)]} + \|e_h u_p\|^2_{L^2[0,T;L^2(\Omega)]} ) +  \|e^N_{h-}\|^2_{L^2(\Omega)} \\
&& \qquad + (1/\epsilon^2) \|e_h\|^4_{L^4[0,T;L^4(\Omega)]}   + \sum_{i=1}^{N-1} \|[e^i_{h}]\|^2_{L^2(\Omega)} \\
&& \leq C (1/\epsilon^6) (\|u_p\|_{L^{\infty}[0,T;L^6(\Omega)]} + \|u\|_{L^{\infty}[0,T;L^6(\Omega)]} )^2 \|e_p\|^2_{L^2[0,T;H^1(\Omega)]} \\
&& \qquad + \frac{1}{\epsilon^6} \|e_p\|^2_{L^2[0,T;L^2(\Omega)]} \Big ).
\end{eqnarray*}
Suppose also that (\ref{eqn:regass}) holds when $k \geq 1$. Then, there exists a constant $C$ depending only upon $\Omega$, and $C_k$ such that
\begin{eqnarray*}
&& \hskip-20pt \|e_h\|^2_{L^{\infty}[0,T;L^2(\Omega)]} \leq  C (1/\epsilon^2) \Big ( (\|u_p\|^{2}_{L^{\infty}[0,T;L^6(\Omega)]} + \|u\|^2_{L^{\infty}[0,T;L^6(\Omega)]} )^2 \|e_p\|^2_{L^2[0,T;H^1(\Omega)]} \\
&& \hskip-20pt \qquad + \|e_p\|^2_{L^2[0,T;L^2(\Omega)]} \Big ).
\end{eqnarray*}
\end{thm}
\begin{proof}
{\it Step 1: Estimate at partition points and in $L^2[0,T;H^1(\Omega)]$:} Since, we have already obtained a bound on $\|e_h\|_{L^2[0,T;L^2(\Omega)]}$ with constant depending polynomially upon $1/\epsilon$, we may return to the orthogonality condition
(\ref{eqn:orth3}) and set $w_h = e_h$. Then, for every $n=1,...,N$, we have:
\begin{eqnarray} \label{eqn:aux1}
&& \frac{1}{2} \|e^n_{h-}\|^2_{L^2(\Omega)} + C\int_{t^{n-1}}^{t^n} \|e_h\|^2_{H^1(\Omega)} dt + \frac{1}{2}\|[e^{n-1}_h]\|^2_{L^2(\Omega)} \nonumber \\
&& \qquad + \frac{1}{\epsilon^2} \int_{t^{n-1}}^{t^n} (\|e_hu_h\|^2_{L^2(\Omega)} + \|e_hu_p\|^2_{L^2(\Omega)} ) dt  + \|[e^{n-1}_h]\|^2_{L^2(\Omega)} \nonumber \\
&& \leq \frac{1}{2} \|e^{n-1}_{h-}\|^2_{L^2(\Omega)}  +  \frac{1}{\epsilon^2} \int_{t^{n-1}}^{t^n} \|e_h\|^2_{L^2(\Omega)} dt \nonumber \\
&& \qquad + \frac{1}{\epsilon^2}  \int_{t^{n-1}}^{t^n} \left (  |(e_p(u^2_p+u^2+u_pu),e_h) |  + |(e_p,e_h)| \right  ) dt. 
\end{eqnarray}
It remains to bound the last two terms:
First, we note that H\"older's and Young's inequalities imply  
\begin{align*}
& \frac{1}{\epsilon^2}  \int_{t^{n-1}}^{t^n}   |(e_p(u^2_p+u^2+u_pu),e_h) | dt \\
&  \leq \frac{C}{\epsilon^2} \int_{t^{n-1}}^{t^n} \|e_p\|_{L^6(\Omega)} (\|u_p\|^2_{L^6(\Omega)} +\|u\|^2_{L^6(\Omega)} ) \|e_h\|_{L^2(\Omega)} dt \\
& \leq \frac{C}{\epsilon^2} (\|u_p\|^2_{L^{\infty}[0,T;L^6(\Omega)]} + \|u\|^2_{L^{\infty}[0,T;L^6(\Omega)]})^2 \int_{t^{n-1}}^{t^n} \|e_p\|^2_{H^1(\Omega)} dt + \frac{C}{\epsilon^2} \int_{t^{n-1}}^{t^n} \|e_h\|^2_{L^2(\Omega)}  dt.
\end{align*}
Substituting the last inequality into (\ref{eqn:aux1}) and summing the resulting inequalities we obtain, 
\begin{eqnarray*}
&& \frac{1}{2} \|e^N_{h-}\|^2_{L^2(\Omega)} + \frac{C}{2} \int_{0}^{T} \|e_h\|^2_{H^1(\Omega)} dt + \frac{1}{2}\sum_{i=1}^{N} \|[e^{i-1}_h]\|^2_{L^2(\Omega)} \nonumber \\
&& \qquad + (1/2\epsilon^2) \int_{0}^{T} (\|e_hu_h\|^2_{L^2(\Omega)} + \|e_hu_p\|^2_{L^2(\Omega)} ) dt   \nonumber \\
&& \leq \frac{C}{\epsilon^2} \int_{0}^{T} \|e_h\|^2_{L^2(\Omega)} dt  + \frac{C}{\epsilon^2} \int_{0}^{T} \|e_p\|^2_{L^2(\Omega} dt \nonumber \\
&& \qquad + \frac{C}{\epsilon^2} (\|u_p\|^2_{L^{\infty}[0,T;L^6(\Omega)]} + \|u\|^2_{L^{\infty}[0,T;L^6(\Omega)]})^2 \int_{0}^{T} \|e_p\|^2_{H^1(\Omega)} dt 
\end{eqnarray*}
It remains to replace the term $ (1/\epsilon^2) \int_{0}^{T} \|e_h\|^2_{L^2(\Omega)} dt$ by Proposition \ref{prop:error}. First, note that  the bound of Proposition \ref{prop:error}, implies that:

\begin{eqnarray*}
&& \frac{1}{2} \|e^N_{h-}\|^2_{L^2(\Omega)} + \frac{C}{2} \int_{0}^{T} \|e_h\|^2_{H^1(\Omega)} dt + \frac{1}{2}\sum_{i=1}^{N} \|[e^{i-1}_h]\|^2_{L^2(\Omega)} \nonumber \\
&& \qquad + (1/2\epsilon^2) \int_{0}^{T} (\|e_hu_h\|^2_{L^2(\Omega)} + \|e_hu_p\|^2_{L^2(\Omega)} ) dt   \nonumber \\
&& \leq \frac{C}{\epsilon^2} (\|u_p\|^2_{L^{\infty}[0,T;L^6(\Omega)]} + \|u\|^2_{L^{\infty}[0,T;L^6(\Omega)]})^2 \int_{0}^{T} \|e_p\|^2_{H^1(\Omega)} dt \\
&& \qquad + \frac{C}{\epsilon^6} ( \|u_p\|_{L^{\infty}[0,T;L^6(\Omega)]} + \|u\|_{L^{\infty}[0,T;L^6(\Omega)]}) )^2 \int_0^T \|e_p\|^2_{H^1(\Omega)} dt \\
&& \qquad + \frac{\delta C}{\epsilon^2} \int_0^T (\|e_hu_h\|^2_{L^2(\Omega)} + \|e_hu_p\|^2_{L^2(\Omega)} ) dt \\
&& \qquad + C \left ( \frac{1}{\epsilon^6} + \frac{1}{\epsilon^2} \|u\|^2_{L^{\infty}[0,T;L^{\infty}(\Omega)]}\right ) \int_0^T \|e_p\|^2_{L^2(\Omega)} dt.
\end{eqnarray*}
Here, $C$ denotes an algebraic constant. 
Now, noting that we may choose $\delta$ in order to hide the $\|e_hu_h\|_{L^2[0,T;L^2(\Omega)]}, \|e_hu_p\|_{L^2[0,T;L^2(\Omega)]}$ which implies the first estimate, after noting that 
since $\|u_p\|_{L^{\infty}[0,T;H^1(\Omega)]} \approx \frac{C}{\epsilon}$ (due to (\ref{eqn:upstab}) and Lemma \ref{lem:contreg})  we may bound 
$$  \frac{C}{\epsilon^2} (\|u_p\|^2_{L^{\infty}[0,T;L^6(\Omega)]} + \|u\|^2_{L^{\infty}[0,T;L^6(\Omega)]})^2 \leq \frac{C}{\epsilon^6}.$$
It is clear that the bounds on $\|e_uu_h\|_{L^2[0,T;L^2(\Omega)}$ and on $\|e_hu_p\|_{L^2[0,T;L^2(\Omega)]}$ imply a similar estimate for $\|e_h\|_{L^4[0,T;L^4(\Omega)]}$, since
\begin{eqnarray*} && (1/\epsilon^2) \int_{0}^{T} \|e_h\|^4_{L^4(\Omega)} dt \leq (2/\epsilon^2) \int_{0}^{T} \int_\Omega |e_h|^2 (|u_h|^2 + |u_p|^2 )dxdt \\
&& \leq   (2/\epsilon^2) \int_{0}^{T} (\|e_hu_h\|^2_{L^2(\Omega)} + \|e_hu_p\|^2_{L^2(\Omega)} )dt.
\end{eqnarray*}
{\it Step 2: Estimates at arbitrary time points:} 
 We proceed to the estimate at arbitrary time-points. We use similar ideas to the proof of Proposition \ref{prop:stabinfty}.
For fixed $t \in [t^{n-1},t^n)$
and $z_h \in U_h$ we set $w_h(s) = z_h \rho (s)$ into
(\ref{eqn:orth3}), with $\rho (s) \in {\mathcal P}_k[t^{n-1},t^n]$ such that
$$ \rho (t^{n-1}) = 1, \qquad \int_{t^{n-1}}^{t^n} \rho q =
\int_{t^{n-1}}^t q, \qquad\, q \in {\mathcal
P}_{k-1}[t^{n-1},t^n].$$
From Lemma\ref{lem:char1} we deduce that $\|\rho\|_{L^{\infty}} \leq C_k$, with $C_k$ independent of
$t$, and 
\begin{eqnarray*}
&& \int_{t^{n-1}}^{t^n} \langle e_{ht},w_h \rangle ds +
(e^{n-1}_{h+}-e^{n-1}_{h-},w^{n-1}_{h+}) \\
&& = \int_{t^{n-1}}^{t}
\langle e_{ht},z_h \rangle ds + (e^{n-1}_{h+}-e^{n-1}_{h-},\rho(t^{n-1}) z_h ) = (e_{h}(t)-e^{n-1}_{h-},z_h).
\end{eqnarray*}
Therefore, integrating by parts (in time),  (\ref{eqn:orth3}), setting $w_h(s) = z_h \rho(s)$, using the above equality and standard algebra, we obtain:
\begin{eqnarray}
&& (e_{h}(t)-e^{n-1}_{h-},z_h)  \leq C_k \Big [ \int_{t^{n-1}}^{t^n}  \int_\Omega |\nabla e_h| |\nabla z_h| dxds  \nonumber \\
&&  +   \frac{1}{\epsilon^2}  \int_{t^{n-1}}^{t^n} \int_\Omega  \left ( |e_h| (|u_h|^2+ |u_p|^2) |z_h|  + |e_h||z_h|  \right )dx ds  \nonumber \\
 && +    \frac{1}{\epsilon^2}  \int_{t^{n-1}}^{t^n} \int_\Omega \left ( |e_p| (|u_p|^2+ |u|^2) |z_h| + |e_p| |z_h|  \right  ) dxds \Big ].     \label{eqn:orth3b} 
\end{eqnarray}
Adding and subtracting $u_p,u$, and using standard algebra, we may bound $$\int_{t^{n-1}}^{t^n}\int_\Omega |e_h| (|u_h|^2 + |u_p|^2) |z_h|dxds \leq C\int_{t^{n-1}}^{t^n} \int_\Omega (|e_h |^3 + |e_h| |u_p-u|^2 + |e_h| |u|^2) |z_h| dxds.$$ Hence, using H\"older's inequality into (\ref{eqn:orth3b}) we derive  
\begin{eqnarray}
&& \langle e_{h}(t)-e^{n-1}_{h-},z_h \rangle \nonumber \\
&& \leq  C_k \Big [ \int_{t^{n-1}}^{t^n} \|\nabla
e_h\|_{L^2(\Omega)} \|\nabla z_h\|_{L^2(\Omega)} +  \frac{1}{\epsilon^2} \int_{t^{n-1}}^{t^n}
\|e_h\|_{L^2(\Omega)} \|z_h\|_{L^2(\Omega)} ds \nonumber \\
&& \qquad + \frac{1}{\epsilon^2} \int_{t^{n-1}}^{t^n} \left ( \|e_h\|^3_{L^{4}(\Omega)} \|z_h\|_{L^4(\Omega)} + \|e_h\|_{L^4(\Omega)} \|e^2_p\|_{L^2(\Omega)}
 \|z_h\|_{L^4(\Omega)}  \right ) dt \nonumber \\
&& \qquad + \frac{1}{\epsilon^2} \int_{t^{n-1}}^{t^n} \|e_h\|_{L^6(\Omega)} \|u^2\|_{L^3(\Omega)} \|z_h\|_{L^2(\Omega)} ds \nonumber \\
&& \qquad +   \frac{1}{\epsilon^2} \int_{t^{n-1}}^{t^n} 
\|e_p\|_{L^{6}(\Omega)} \|u^2_p+u^2+u_pu\|_{L^3(\Omega)} \|z_h\|_{L^2(\Omega)} ds  \nonumber \\
&& \qquad +   \frac{1}{\epsilon^2} \int_{t^{n-1}}^{t^n}
\|e_p\|_{L^2(\Omega)} \|z_h\|_{L^2(\Omega)} ds \Big ]. \label{eqn:aux2}
\end{eqnarray}
Noting that $z_h$ is independent of $t$, and standard algebra implies that 
\begin{eqnarray}
&& \langle e_{h}(t)-e^{n-1}_{h-},z_h \rangle \nonumber \\
&& \leq  C_k  \Big [ \|z_h\|_{H^1(\Omega)} \int_{t^{n-1}}^{t^n} \|\nabla
e_h\|_{L^2(\Omega)} ds +  \frac{1}{\epsilon^2}\|z_h\|_{L^2(\Omega)} \int_{t^{n-1}}^{t^n}
\|e_h\|_{L^2(\Omega)} ds \nonumber \\
&& \qquad +  \frac{1}{\epsilon^2} \|z_h\|_{L^4(\Omega)} \int_{t^{n-1}}^{t^n}  \|e_h\|^3_{L^4(\Omega)} ds 
+ \frac{1}{\epsilon^2} \|z_h\|_{L^4(\Omega)} \int_{t^{n-1}}^{t^n} \|e_h\|_{L^4(\Omega)} \|e_p\|^2_{L^4(\Omega)}  ds \nonumber  \\
&& \qquad + \frac{1}{\epsilon^2} \|z_h\|_{L^2(\Omega)} 
\int_{t^{n-1}}^{t^n} \|e_h\|_{L^6(\Omega)} \|u\|^2_{L^6(\Omega)} ds  \nonumber  \\
&& \qquad +   \frac{1}{\epsilon^2} \|z_h\|_{L^2(\Omega)} \int_{t^{n-1}}^{t^n} \|e_p\|_{L^6(\Omega)} (\|u_p\|^2_{L^6(\Omega)}+\|u\|^2_{L^6(\Omega)}) ds \nonumber \\
&& \qquad +  \frac{1}{\epsilon^2} \|z_h\|_{L^2(\Omega)} \int_{t^{n-1}}^{t^n}
\|e_p\|_{L^2(\Omega)} ds \Big ]. \label{eqn:aux3}
\end{eqnarray}
Using once more H\"older's inequality and the fact that $u,u_p \in L^{\infty}[0,T;H^1(\Omega)]$,  we deduce with different constant $C_k$ (independent of $\epsilon$): 
\begin{align*}  
&  \langle e_{h}(t)-e^{n-1}_{h-},z_h \rangle \leq  C_k  \Big [ \|z_h\|_{H^1(\Omega)} \tau^{1/2}_n \|\nabla
e_h\|_{L^2[t^{n-1},t^n;L^2(\Omega)]}  \\
& + \frac{\tau^{1/2}_n}{\epsilon^2} \|z_h\|_{L^2(\Omega)}  
\|e_h\|_{L^2[t^{n-1};t^n;L^2(\Omega)]} + \frac{\tau^{1/4}_n}{\epsilon^2} \|z_h\|_{L^4(\Omega)} \|e_h\|^3_{L^4[t^{n-1},t^n;L^4(\Omega)]} \\
& + \frac{ \tau^{1/4}_n}{\epsilon^2} \|z_h\|_{L^4(\Omega)} 
\|e_h\|_{L^4[t^{n-1},t^n;L^4(\Omega)]} \|e_p\|^2_{L^{4}[0,T;L^4(\Omega)]} \\
& + \frac{\tau^{1/2}_n}{\epsilon^2} \|z_h\|_{L^2(\Omega)} \|e_h\|_{L^2[t^{n-1},t^n;L^6(\Omega)]} \|u\|^2_{L^{\infty}[0,T;L^6(\Omega)]}\\
& +  \frac{\tau^{1/2}_n}{\epsilon^2} \|z_h\|_{L^2(\Omega)}  \|e_p\|_{L^2[t^{n-1};t^n;H^1(\Omega)]}  (\|u_p\|^2_{L^{\infty}[0,T;L^6(\Omega)]} +
 \|u\|^2_{L^{\infty}[0,T;L^6(\Omega)]} ) \\
& +  \frac{\tau^{1/2}_n}{\epsilon^2} \|z_h\|_{L^2(\Omega)} 
\|e_p\|_{L^2[t^{n-1},t^n;L^2(\Omega)]} \Big ]. 
\end{align*}
Setting $z_h = e_h(t)$ and integrating with respect to time, using H\"older's inequallity to bound $\int_{t^{n-1}}^{t^n} \|e_h(t)\|_{L^4(\Omega)} dt \leq \tau^{3/4} \|e_h\|_{L^4[t^{n-1},t^n;L^4(\Omega)]}$, 
and standard calculations, we derive,
\begin{align} \label{eqn:aux4}
& \int_{t^{n-1}}^{t^n}
\|e_h(t)\|^2_{L^2(\Omega)} dt
\leq  \|e^{n-1}_{h-}\|_{L^2(\Omega)} \tau^{1/2}_n \|e_h(t)\|_{L^2[t^{n-1},t^n;L^2(\Omega)]}   \\
&+ C_k \Big [ \tau_n \|e_h\|^2_{L^2[t^{n-1},t^n;H^1(\Omega)]} +\frac{\tau_n}{\epsilon^2} \|e_h\|^4_{L^4[t^{n-1},t^n;L^4(\Omega)]} \nonumber  \\
&+  \frac{\tau_n}{\epsilon^2} \|e_h\|^2_{L^4[t^{n-1};t^n;L^4(\Omega)]} \|e_p\|^2_{L^{4}[t^{n-1},t^n;L^4(\Omega)]}    
\nonumber \\
& + \frac{\tau_n}{\epsilon^2} \|e_h\|_{L^2[t^{n-1};t^n;L^2(\Omega)]} \|e_h\|_{L^{2}[t^{n-1},t^n;H^1(\Omega)]}   \|u\|^2_{L^{\infty}[0,T;L^6(\Omega)]}  
\nonumber \\
&+\frac{\tau_n}{\epsilon^2} \|e_h\|_{L^2[t^{n-1},t^n;L^2(\Omega)]} \|e_p\|_{L^2[t^{n-1},t^n;H^1(\Omega)]}   (\|u_p\|^2_{L^{\infty}[0,T;L^6(\Omega)]} + 
\|u\|^2_{L^{\infty}[0,T;L^6(\Omega)]} ) \nonumber \\
& + \frac{\tau_n}{\epsilon^2} \|e_h\|_{L^2[t^{n-1},t^n;L^2(\Omega)]} \|e_p\|_{L^2[t^{n-1},t^n;L^2(\Omega)]} \Big ]. \nonumber 
\end{align}
For the first term of the left hand side, using Young's inequality, we obtain: 
\begin{align*}
& \|e^{n-1}_{h-}\|_{L^2(\Omega)} \tau^{1/2}_n \|e_h(t)\|_{L^2[t^{n-1},t^n;L^2(\Omega)]} \\
& \quad  \leq \frac{1}{4} 
 \|e_h(t)\|^2_{L^2[t^{n-1},t^n;lL^2(\Omega)]}  + C \tau_n \|e^{n-1}_{h+}\|^2_{L^2(\Omega)}.
\end{align*}
For the fourth term, we note that using Young's inequality, we obtain
\begin{align*}
& \frac{\tau_n}{\epsilon^2} \|e_h\|^2_{L^4[t^{n-1},t^n;L^4(\Omega)]}   \|e_p\|^2_{L^4[t^{n-1},t^n;L^4(\Omega)]}  \\
& \leq \frac{\tau_n}{\epsilon^2} \|e_h\|^4_{L^4[t^{n-1},t^n;L^4(\Omega)]}  + \frac{\tau_n}{\epsilon^2} \|e_p\|^4_{L^4[t^{n-1},t^n;L^4(\Omega)]}.  
\end{align*}
For the fifth term, we note that (\ref{eqn:regass}), implies that $\|u\|_{L^{\infty}[0,T;L^6(\Omega)]} \leq C$ (where $C$ is independent of $\epsilon$), and hence we obtain,
\begin{align*}
& \frac{\tau_n}{\epsilon^2} \|e_h\|_{L^2[t^{n-1},t^n;L^2(\Omega)]} \|e_h\|_{L^2[t^{n-1},t^n;H^1(\Omega)]} \|u\|^2_{L^{\infty}[0,T;L^6(\Omega)]} \\
& \quad \leq \frac{\tau_n}{\epsilon^4} \|e_h\|^2_{L^2[t^{n-1},t^n;L^2(\Omega)]} + \tau_n \|e_h\|^2_{L^2[t^{n-1},t^n;H^1(\Omega)]}. 
\end{align*}
For the last two terms, using similar algebra, we deduce, 
\begin{align*}
& \frac{\tau_n}{\epsilon^2} \|e_h\|_{L^2[t^{n-1},t^n;L^2(\Omega)]} \|e_p\|_{L^2[t^{n-1},t^n;H^1(\Omega)]}   (\|u_p\|^2_{L^{\infty}[0,T;L^6(\Omega)]} + 
\|u\|^2_{L^{\infty}[0,T;L^6(\Omega)]} )  \\
& \quad \leq \frac{\tau_n}{\epsilon^4} \|e_h\|^2_{L^2[t^{n-1},t^n;L^2(\Omega)]} \\
& \qquad + \tau_n  \|e_p\|^2_{L^2[t^{n-1},t^n;H^1(\Omega)]} \times (\|u_p\|^2_{L^{\infty}[0,T;L^6(\Omega)]} 
+ \|u\|^2_{L^{\infty}[0,T;L^6(\Omega)]} )^2, \\
&  \frac{\tau_n}{\epsilon^2} \|e_h\|_{L^2[t^{n-1},t^n;L^2(\Omega)]} \|e_p\|_{L^2[t^{n-1},t^n;L^2(\Omega)]} \\
& \quad \leq \frac{\tau_n}{\epsilon^4}  \|e_h\|^2_{L^2[t^{n-1},t^n;L^2(\Omega)]} + \tau_n  \|e_p\|^2_{L^2[t^{n-1},t^n;L^2(\Omega)]}.
\end{align*}
Note that choosing $\frac{C_k\tau_n}{\epsilon^4} \leq \frac{1}{8}$, we may hide all $\|e_h\|^2_{L^2[t^{n-1},t^n;L^2(\Omega)]}$ of (\ref{eqn:aux4}) on the left. Hence, dividing by $\tau_n$ the resulting inequality and using an inverse estimate in time, we arrive at,
\begin{align*}
& \|e_h\|^2_{L^{\infty}[t^{n-1},t^n;L^2(\Omega)]}  \leq C_k \Big ( \|e^{n-1}_{h+}\|^2_{L^2(\Omega)} + \|e_h\|^2_{L^2[t^{n-1},t^n;H^1(\Omega)]} + \frac{1}{\epsilon^2} \|e_h\|^4_{L^4[t^{n-1},t^n;L^4(\Omega)]} \\
& \quad + \|e_p\|^2_{L^2[t^{n-1},t^n;H^1(\Omega)]}   (\|u_p\|^2_{L^{\infty}[0,T;L^6(\Omega)]} + 
\|u\|^2_{L^{\infty}[0,T;L^6(\Omega)]} )^2 + \|e_p\|^2_{L^2[t^{n-1},t^n;L^2(\Omega)]} \\
& \quad + \frac{1}{\epsilon^2} \|e_p\|^4_{L^4[t^{n-1},t^n;L^4(\Omega)]}   \Big ).
\end{align*}
Now, note that $$\|e_p\|^4_{L^4[t^{n-1},t^n;L^4(\Omega)]} \leq ( \|u_p\|^2_{L^{\infty}[t^{n-1},t^n;L^4(\Omega)]} + \|u\|^2_{L^{\infty}[t^{n-1},t^n;H^1(\Omega)]}) \|e_p\|^2_{L^2[t^{n-1},t^n;H^1(\Omega)]}.$$ Hence, the desired estimate now follows by replacing the bounds of $\|e^{n-1}_{h+}\|^2_{L^2(\Omega)}$, $\frac{1}{\epsilon^2} \|e_h\|^4_{L^4[0,T;L^4(\Omega)]}$, $\|e_h\|^2_{L^2[0,T;H^1(\Omega)]}$.
\end{proof}
\begin{rmk} \begin{enumerate}
 \item The estimate at arbitrary time points results in a best-approximation result by using triangle inequality. In addition, the dependence of the constant upon $\frac{1}{\epsilon}$ doesn't deteriorate further, despite the fact that we treat schemes of arbitrary order, provided that the natural assumption  $\frac{1}{\epsilon^2}\|(u^2_0-1)^2\|_{L^1(\Omega)} \leq C$ holds. 
\item Our estimate at the energy norm at partition points is valid even without assuming the bound $\frac{1}{\epsilon^2}\|(u^2_0-1)^2\|_{L^1(\Omega)} \leq C$ (with $C$ independent of $\epsilon$). 
\end{enumerate}
\end{rmk}
The best approximation estimate now follows by triangle inequality.
\begin{thm} \label{thm:symm}
 Suppose that (\ref{eqn:regass}) holds.  Let $\tau \leq \tilde \tau$, $h \leq \tilde h$, (where $\tilde \tau$, $\tilde h$ defined by (\ref{eqn:conv})), and in addition let 
\begin{enumerate} 
\item $\tau +h \leq \frac{\delta C\epsilon^4}{(\|u\|_{L^2[0,T;H^2(\Omega)]} + \|u_t\|_{L^2[0,T;H^1(\Omega)]})}$, when $d=3$, 
\item $\tau +h \leq \frac{\delta C\epsilon^{7/2}}{(\|u\|_{L^2[0,T;H^2(\Omega)]} + \|u_t\|_{L^2[0,T;H^1(\Omega)]})}$, when $d=2$.
\item $\tau + h \leq \frac{ \delta C \epsilon^3}{(\|u\|_{L^2[0,T;H^2(\Omega)]} + \|u_t\|_{L^2[0,T;H^1(\Omega)]})}$ when $d=2$, $k=0,1$.
\end{enumerate}
Then, there exists a constant ${\bf C}$ depending only upon $\Omega$, $C_k$  and $ \|u_p\|_{L^{\infty}[0,T;L^6(\Omega)]}+ \|u\|_{L^{\infty}[0,T;L^6(\Omega)]}$
but independent of $\epsilon$, and the such that, 
\begin{eqnarray*}
&& \hskip-20pt \|e\|_{L^2[0,T;H^1(\Omega)]}  + \|e\|_{L^{\infty}[0,T;L^2(\Omega)]} \leq {\bf C}(1/\epsilon^3) \left ( \|e_p\|_{L^2[0,T;H^1(\Omega)]} + \|e_p\|_{L^{\infty}[0,T;L^2(\Omega)]} \right ).
\end{eqnarray*}
If in addition $u \in L^2[0,T;H^{l+1}(\Omega)]$,  $u^{(k+1)} \in L^{\infty}[0,T;L^2(\Omega)]$
there exists a positive constant ${C}$ that depends only upon $\Omega,C_k$ and it is independent of $h,\tau,\epsilon$,  such that
\begin{align*}
& \|e\|_{L^2[0,T;H^1(\Omega)]} + \|e\|_{L^{\infty}[0,T;L^2(\Omega)]} \\
& \leq {C} (1/\epsilon^4) \big (
h^{l} \|u\|_{L^2[0,T;H^{l+1}(\Omega)]} + \tau^{k+1} \|u^{(k+1)}\|_{L^{\infty}[0,T;L^2(\Omega)]} \big ).
\end{align*}
\end{thm}
\begin{proof}
 Using triangle inequality we obtain the first estimate. Then, the rates of convergence follow by the estimates on $e_p$ in $L^2[0,T;H^1(\Omega)]$, and $L^{\infty}[0,T;L^2(\Omega)]$ norms using Lemma \ref{lem:rates1} and \ref{eqn:projest1}.  since  $\|u_p\|_{L^{\infty}[0,T;H^1(\Omega)]} + \|u\|_{L^{\infty}[0,T;H^1(\Omega)]} \leq C/\epsilon$ by (\ref{eqn:upstab}).
\end{proof}
\begin{prop} \label{prop:rates}
Let $k=0$, $l=1$.  If $u \in L^2[0,T;H^2(\Omega)] \cap H^1[0,T;L^2(\Omega)]$  suppose that $\tau,h$ satisfy $\tau^{1/2} + h \leq \frac{C\epsilon^{2}}{\|u_t\|_{L^2[0,T;L^2(\Omega)]} + \|u\|_{L^2[0,T;H^2(\Omega)]}}$ for $d=2$, and $\tau^{1/2} +h \leq \frac{C\epsilon^{8/3}}{\|u_t\|_{L^2[0,T;L^2(\Omega)]}+\|u\|_{L^2[0,T;H^2(\Omega)]}}$ for $d=3$. If $u \in L^2[0,T;H^2(\Omega)] \cap H^1[0,T;H^1(\Omega)]$ suppose that $\tau+ h \leq \frac{C\epsilon^{3}}{\|u_t\|_{L^2[0,T;H^1(\Omega)]} + \|u\|_{L^2[0,T;H^2(\Omega)]}}$ for $d=2$, and $\tau +h \leq \frac{C\epsilon^{4}}{\|u_t\|_{L^2[0,T;H^1(\Omega)]} + \|u\|_{L^2[0,T;H^2(\Omega)]}}$ for $d=3$. Then, there exists a positive constant ${\bf C}$ depending only upon $\Omega$, $C_k$  and $ \|u_p\|^2_{L^{\infty}[0,T;L^6(\Omega)]} + \|u\|^2_{L^{\infty}[0,T;L^6(\Omega)]}$
but independent of $h,\tau,\epsilon$ such that, 
\begin{enumerate}
\item $\|e\|_{L^2[0,T;H^1(\Omega)]} + \|e\|_{L^{\infty}[0,T;L^2(\Omega)]} \leq {\bf C}(1/\epsilon^3) (\tau^{1/2}  + h)$, \\
\quad\mbox{ when $u \in L^2[0,T;H^2(\Omega)] \cap H^1[0,T;L^2(\Omega)]$,} 
\item $\|e\|_{L^2[0,T;H^1(\Omega)]} + \|e\|_{L^{\infty}[0,T;L^2(\Omega)]} \leq {\bf C}(1/\epsilon^3) (\tau+ h)$,  \\
 \quad\mbox{when $u \in L^2[0,T;H^2(\Omega)] \cap H^1[0,T;H^1(\Omega)]$.}
\end{enumerate}
\end{prop}
\begin{proof} The estimates concerning the lowest order scheme follow directly from Theorem 5.7, and the approximation properties of $e_p$ in $L^2[0,T;H^1(\Omega)]$ and $L^2[0,T;L^2(\Omega)]$ norms, when $u \in L^2[0,T;H^2(\Omega)] \cap H^1[0,T;H^1(\Omega)]$. When $u \in L^2[0,T;H^2(\Omega)] \cap H^1[0,T;L^2(\Omega)]$ then the time step and spacial discretization size restrictions are replaced by the ones of Remark 5.6.
\end{proof}

We close this section by discussing the discrete analog of the energy conservation property.
\begin{rmk}
Given initial data $u_0 \in H^1_0(\Omega)$, and zero forcing term $f=0$, it is well known that the solution of (\ref{eqn:ac}) satisfies, for any $t \geq 0$,
\begin{equation} \label{eqn:en2} \frac{d}{dt} E(t) + \|u_t(t)\|^2_{L^2(\Omega)} = 0 
\end{equation}
where $E(t)$ denotes the associated energy i.e.,
$$E(t) = \int_{\Omega} \left ( \frac{1}{2} |\nabla u|^2 + \frac{1}{4\epsilon^2} (u^2-1)^2  \right ) dx.$$
It is clear that the discrete solution of (\ref{eqn:dac}) does not possess any meangingful regularity for $u_{ht}$, due to the discontinuities in time and hence (\ref{eqn:en2}) is not valid by simply replacing $u$ by $u_h$. However, for any $t \in (t^{n-1},t^n]$, we may formally rewrite (\ref{eqn:en2}) as,
\begin{equation*} (t-t^{n-1}) \frac{d}{dt} E(t) + (t-t^{n-1}) \|u_t(t)\|^2_{L^2(\Omega)} = 0 \end{equation*}
and hence integrating with respect to time and using integration parts in time,
\begin{equation} \label{eqn:en3}
\tau_n E(t^n) - \int_{t^{n-1}}^{t^n} E(t) dt + \int_{t^{n-1}}^{t^n} (t-t^{n-1}) \|u_t(t)\|^2_{L^2(\Omega)} = 0.
\end{equation}
It is clear now that the above equality (\ref{eqn:en3}) is well defined, and (at least formally) we may replace $u$ by any $u_h \in {\mathcal P}_{k}[t^{n-1},t^n;U_h]$. 

We observe that integrating by parts (in time), (\ref{eqn:dac}) and setting $v_h = (t-t^{n-1}) u_{ht} \in {\mathcal P}_k[t^{n-1},t^n;U_h]$, with $k \geq 1$, we obtain
\begin{eqnarray*}
&& \int_{t^{n-1}}^{t^n} (t-t^{n-1}) \|u_{ht}\|^2_{L^2(\Omega)} dt \\
&& \qquad + \int_{t^{n-1}}^{t^n} \left ( (t-t^{n-1}) \frac{d}{dt} \left ( \|\nabla u_h(t)\|^2_{L^2(\Omega)} + \frac{1}{4 \epsilon^2} \|(u^2_h-1)^2\|_{L^1(\Omega)} \right ) \right ) dt = 0,
\end{eqnarray*}
which implies (after integration by parts in time for the second integral) 
\begin{eqnarray*}
&& \int_{t^{n-1}}^{t^n} (t-t^{n-1}) \|u_{ht}\|^2_{L^2(\Omega)} dt + \tau_n \left ( \|\nabla u^n_{h-}\|^2_{L^2(\Omega)} + \frac{1}{4 \epsilon^2} \|((u^2)^{n}_{h-} - 1)^2\|_{L^1(\Omega)}  \right ) \\
&& - \int_{t^{n-1}}^{t^n}  \left ( \|\nabla u_h(t)\|^2_{L^2(\Omega)} + \frac{1}{4 \epsilon^2} \|(u^2_h-1)^2\|_{L^1(\Omega)}  \right ) dt = 0.
\end{eqnarray*}
Hence, we have shown that the discrete solution constructed by $(\ref{eqn:dac})$ actually satisfies a discrete local analog of the energy equality. 
It remains to prove that $\|\nabla u^n_{h-}\|_{L^2(\Omega)}$ and $ \|u^n_{h-}\|_{L^4(\Omega)}$ are also bounded, independent of $\tau,h$, which is easily obtained by using the results of Theorem \ref{thm:est} and an inverse estimate. Indeed, recall that under the assumptions of  Theorem \label{thm:est}, we deduce, for any $u \in L^2[0,T;H^2(\Omega)] \cap H^1[0,T;H^1(\Omega)]$ for any $\tau \leq Ch$,
\begin{eqnarray*}
&& \|u_h\|_{L^{\infty}[0,T;H^1(\Omega)]} \leq C \frac{1}{h} \|u_h-u_p\|_{L^{\infty}[0,T;L^2(\Omega)]} + C \|u_p\|_{L^{\infty}[0,T;H^1(\Omega)]} \\
&&  \leq  {\bf C} (1/\epsilon^2) \left (  \frac{\tau}{h} + 1 \right ).
\end{eqnarray*}
\end{rmk}

{\it {\bf Acknowlegdement:}}
The author would like to thank the referees for many valuable comments and suggestions. In particular, we thank a referee for pointing out a gap in a previous version of the proof of Proposition 5.5

\font \smcb=cmr6

\end{document}